\documentclass[11pt]{amsart}
\usepackage[margin=1.4in]{geometry}
\usepackage{amssymb}
\usepackage{amssymb,amsmath,amsthm,mathtools}
\usepackage{physics}
\usepackage{float}
\usepackage[english]{babel}
\usepackage{hhline}

\newtheorem{thm}{Theorem}[section]
\newtheorem{prop}[thm]{Proposition}
\newtheorem{lem}[thm]{Lemma}
\newtheorem{cor}[thm]{Corollary}
 
\theoremstyle{definition}
\newtheorem{example}[thm]{Example}
\newtheorem{remark}[thm]{Remark}
\newtheorem{definition}[thm]{Definition}

\numberwithin{equation}{section}

\newcommand{\wt}{\widetilde}

\newcommand{\PGL}{\operatorname{PGL}}
\newcommand{\Aut}{\operatorname{Aut}}

\newcommand{\id}{\operatorname{id}}
\newcommand{\Pic}{\operatorname{Pic}}

\newcommand{\im}{\operatorname{im}}

\newcommand{\Gal}{\operatorname{Gal}}

\newcommand{\beq}{\begin{equation}}
\newcommand{\eeq}{\end{equation}}
\usepackage[
            pdftex,
           colorlinks=true,
            urlcolor=blue,       % \href{...}{...} external (URL)
            filecolor=blue,     % \href{...} local file
            linkcolor=blue,       % \ref{...} and \pageref{...}
            citecolor=blue,         % \cite{...}
            pdftitle= {Title},
            pdfauthor={Author},
            pdfsubject={Subject},
            pdfkeywords={Key}
            pagebackref,
            bookmarksopen=true]
            {hyperref}
\usepackage{hyperref}

\begin{document}

\title{Automorphisms of Quartic del Pezzo Surfaces in Characteristic Zero}
\author{Jonathan M. Smith}

\maketitle

\abstract{}
For each field $k$ of characteristic zero, we classify which groups act by automorphisms on a quartic del Pezzo surface over $k$. We also determine which groups act on $k$-rational, stably $k$-rational, or $k$-unirational quartic del Pezzo surfaces. For each group $G$ that is realized over $k$, we exhibit explicit equations for a quartic del Pezzo surface $X$ in $\mathbb{P}^4_k$ such that $G$ acts by automorphisms on $X$.
\endabstract{}

\section{Introduction}

Given a group $G$ and a field $k$, does there exist a quartic del Pezzo surface $X$ over $k$ such that $G \xhookrightarrow{} \Aut(X)$? The primary purpose of this paper is to answer this question for every group and every field of characteristic zero. We also determine which groups act on a $k$-rational, stably $k$-rational, or $k$-unirational quartic del Pezzo surface. Determining which groups act on a $k$-rational quartic del Pezzo surface is progress toward the larger goal of classifying the finite subgroups of the plane Cremona group over $k$.

The automorphism groups of del Pezzo surfaces of any degree have been studied extensively when $k = \bar{k}$. If $k$ is a field of characteristic zero and $k = \bar{k}$, the possible automorphism groups of a del Pezzo surface of degree $d$ over $k$ can be found in \cite{DI09}. The classification over algebraically closed fields of positive characteristic was recently completed due to I. Dolgachev, A. Duncan, and G. Martin (see \cite{DD18} for degrees 3 and 4; see \cite{DM23} and \cite{DM232} for degrees 1 and 2). 

Much less is known when $k$ is not algebraically closed. The automorphism groups of $\mathbb{R}$-rational del Pezzo surfaces were determined by E. Yasinsky \cite{Yasinsky}. When $k$ is an arbitrary perfect field, $k$-rational del Pezzo surfaces of degree 8 and degree 6 are classified in \cite{SZ21}. More recently, A. Zaitsev \cite{Z23} and A. Boitrel \cite{B23} independently described the potential automorphism groups of del Pezzo surfaces of degree five over an arbitrary perfect field.

Our results are not strictly an extension of \cite{B23} and \cite{Z23}. We do not exhaustively classify the full automorphism groups of quartic del Pezzo surfaces (as \cite{B23} and \cite{Z23} do for degree 5). Instead, after fixing the field, we determine the maximal automorphism groups which act on \textit{some} quartic del Pezzo surface. To emphasize the distinction, notice that our classification is trivial for del Pezzo surfaces of degree 5. For any field $k$, there is an injection $G \xhookrightarrow{} \Aut(X)$ for some quintic del Pezzo surface $X$ if and only if $G \subseteq S_5$.

Suppose $k$ is a field of characteristic zero. If $X$ is a quartic del Pezzo surface over $k$ there is an injective group homomorphism $\Aut(X) \to W(\mathsf{D}_5)$, where $\Aut(X)$ is the group of automorphisms of $X$, and $W(\mathsf{D}_5)$ is the Weyl group of the root system $\mathsf{D}_5$. This map is induced by the action of $\Aut(X)$ on the 16 exceptional divisors of $\overline{X}$ after picking a geometric basis for $\Pic \overline{X}$. A different choice of geometric basis will produce a conjugate subgroup of $W(\mathsf{D}_5)$. In this way, $W(\mathsf{D}_5)$ offers an ambient space in which we can compare automorphism groups of different quartic del Pezzo surfaces across different fields. In particular, a subgroup $G \subseteq W(\mathsf{D}_5)$ --- up to conjugacy --- acts by automorphisms on $X$ when $G$ is contained in the image of the map $\Aut(X) \to W(\mathsf{D}_5)$. 

Let $\mathcal{M}_k$ denote the collection of conjugacy classes of subgroups of $W(\mathsf{D}_5)$ that act by automorphisms on some quartic del Pezzo surface over $k$. We say a group $G$ in $\mathcal{M}_k$ is maximal if it is maximal with respect to inclusion. Our main result is the following theorem.
\begin{thm}
    \label{thm:main_intro}
    Let $k$ be a field of characteristic zero. If $G$ is a maximal group in $\mathcal{M}_k$, then $G$ is one of the groups in the table. Each group appears in $\mathcal{M}_k$ if and only if $k$ satisfies the condition in the second column.
    \begin{center}
    \begin{tabular}{|c|c|}
        \hline
        \textbf{Group} & \textbf{Condition on } $k$ \\
        \hline
        $C_2^4 \rtimes C_4$ & $i \in k$ \\
        \hline
        $C_2^4 \rtimes S_3$ & $\epsilon_3 \in k$ \\
        \hline
        $C_2^4 \rtimes D_5$ & $\sqrt{5} \in k$ \\
        \hline
        $C_2^4 \rtimes C_2$ & none \\
        \hline
        $C_2^3 \rtimes S_3$ & none \\
        \hline
        $C_2^3 \cdot S_3$ & $x^2 + y^2 = -3$ has a solution over $k$ \\
        \hline
    \end{tabular}
    \end{center}
\end{thm}

\begin{remark}
    The groups listed in Theorem \ref{thm:main_intro} and Theorem \ref{thm:intro2} (see below) correspond to specific conjugacy classes of subgroups in $W(\mathsf{D}_5)$. See Table \ref{table:classes} for a list of these classes. The notation $A \cdot B$ denotes an extension of $B$ by $A$. The groups $C_2^3 \cdot S_3$ and $C_2 \cdot S_3$ in Theorems \ref{thm:main_intro} and \ref{thm:intro2} are specific non-split extensions. The first three groups listed in the table of Theorem \ref{thm:main_intro} are the maximal groups in $\mathcal{M}_{\bar{k}}$, where $\bar{k}$ is an algebraically closed field. The maximal groups in $\mathcal{M}_{\bar{k}}$ are well known (see \cite{D12}, \cite{DI09}). 
\end{remark}

\begin{remark}
    Theorem \ref{thm:main_intro} determines the groups $G \subseteq W(\mathsf{D}_5)$ that act on a quartic del Pezzo surface over $k$ in the following way: $\mathcal{M}_k$ is precisely the downward closure in $W(\mathsf{D}_5)$ of the groups in the table of Theorem \ref{thm:main_intro} that are contained in $\mathcal{M}_k$. 
\end{remark}

Let $\mathcal{M}_k^{rat}$ denote the collection of groups in $W(\mathsf{D}_5)$ that act by automorphisms on a $k$-rational quartic del Pezzo surface over $k$. Our second main result classifies which groups act on a $k$-rational quartic del Pezzo surface.

\begin{thm}
    \label{thm:intro2}
    Let $k$ be a field of characteristic zero.
    \begin{itemize}
        \item[(i)] If $\epsilon_3 \in k$, then $\mathcal{M}_k = \mathcal{M}_k^{rat}$.

        \item[(ii)] If $x^2 + y^2 = -3$ does not have a solution over $k$, then $\mathcal{M}_k = \mathcal{M}_k^{rat}$.

        \item[(iii)] If $x^2+y^2 = -3$ has a solution over $k$ and $\epsilon_3 \not\in k$, then every group in $\mathcal{M}_k$ is in $\mathcal{M}_k^{rat}$ except the groups $C_2^3 \cdot S_3$ and $C_2 \cdot S_3$.
    \end{itemize}
\end{thm}

\begin{remark}
    Take special note of the last group listed in Theorem \ref{thm:main_intro}. The condition that $x^2 + y^2 = -3$ has a solution over $k$ is equivalent to the splitting condition for the quaternion algebra $k\langle i,j \mid i^2 = -1, j^2 = -3, ij = -ji \rangle$. This condition is sufficient for $C_2^3 \cdot S_3$ to act by automorphisms on some quartic del Pezzo surface over $k$, but it is not sufficient for $C_2^3 \cdot S_3$ to act on a $k$-rational quartic del Pezzo surface. The group $C_2 \cdot S_3$ is a subgroup of $C_2^3 \cdot S_3$ that exhibits this same phenomenon. However, these are the only subgroups of $W(\mathsf{D}_5)$ up to conjugacy that appear in some $\mathcal{M}_k$ without acting on a $k$-rational surface. In fact, every other subgroup of $W(\mathsf{D}_5)$ that appears in $\mathcal{M}_k$ acts on a quasi-split surface over $k$ (see Definition \ref{def:quasi-split} and Theorem \ref{Theorem:Non-quasi}).
\end{remark}

This paper is organized as follows. In Section 2, we recall some facts about the structure of the exceptional divisors on a quartic del Pezzo surface and the action of the automorphism group. In Section 3, we prove an intermediate result classifying the groups that act by automorphisms on quasi-split quartic del Pezzo surfaces. In Section 4, we extend this intermediate result to find the groups that act on any quartic del Pezzo surface. In Section 5, we exhibit explicit equations for quartic del Pezzo surfaces in $\mathbb{P}^4_k$ realizing the automorphism groups in Theorem \ref{thm:main_intro} over the appropriate field. In Section 6, we determine which groups act on a $k$-rational, stably $k$-rational, or $k$-unirational quartic del Pezzo surface.

\subsection*{Acknowledgments} I would like to thank my advisor, Alexander Duncan, for many helpful conversations. The author was partially supported by a USC SPARC Grant.

\section{Preliminaries}

Throughout the paper, unless explicitly stated otherwise, $k$ will be an arbitrary field of characteristic zero. We pick an algebraic closure $\bar{k}$ and denote the absolute Galois group $\Gal(\bar{k}/k)$ as $\Gamma$. If $X$ is a smooth projective surface over $k$, we denote $X \times_{\text{Spec}(k)} \text{Spec}(\bar{k})$ as $\overline{X}$. There is an induced action of $\Gamma$ on $\overline{X}$, and there is also a natural action of $\Gamma$ on the group of automorphisms $\Aut(\overline{X})$ of $\overline{X}$. The group of automorphisms of $X$ defined over $k$, denoted $\Aut(X)$, is precisely the group of $\bar{k}$-automorphisms fixed by the $\Gamma$-action. 

Given surfaces $X$ and $Y$ over $k$, a rational map $f: X \dashrightarrow Y$ is defined over $k$ when the corresponding map $\bar{f}: \overline{X} \dashrightarrow \overline{Y}$ is $\Gamma$-equivariant. The surface $X$ is said to be \textit{$k$-rational} if there exists a birational map $f: X \dashrightarrow \mathbb{P}^2_k$ defined over $k$. The surface $X$ is \textit{stably $k$-rational} if $X \times \mathbb{P}^m_k$ is $k$-rational for some $m$. The surface $X$ is \textit{$k$-unirational} if there exists a dominant rational map $f: \mathbb{P}^m_k \dashrightarrow X$ defined over $k$.

\subsection{Group theoretic notation} Throughout, we use the following notation:
\begin{itemize}
    \item[-] $A \rtimes B$ is a semidirect product where $A$ is normal;
    \item[-] $A \cdot B$ denotes an extension of $B$ by $A$;
    \item[-] $C_n$ is the cyclic group with $n$ elements;
    \item[-] $S_n$ is the symmetric group on $n$ letters;
    \item[-] $D_n$ is the dihedral group with $2n$ elements;
    \item[-] $W(\mathsf{D}_5)$ is the Weyl group of the root system $\mathsf{D}_5$;
    \item[-] $\epsilon_n$ denotes a primitive $n$th root of unity;
\end{itemize}
The group $W(\mathsf{D}_5)$ will play an integral role, so we describe its structure here. $W(\mathsf{D}_5)$ is isomorphic to $C_2^4 \rtimes S_5$. If we set
\begin{equation*}
    C_2^4 = \left\{(a_1,a_2,a_3,a_4,a_5) \in C_2^5 \mid \sum a_i = 0 \right\}
\end{equation*}
then $S_5$ acts on $C_2^4$ by permuting the coordinates. For the remainder of the paper we will denote elements of $W(\mathsf{D}_5)$ via the identification
\begin{equation*}
    W(\mathsf{D}_5) = \left\{ab \bigm| a = (a_i)_{i=1}^5 \in C_2^5 \text{ with } \sum_{i=1}^5 a_i = 0, \text{ and } b \in S_5 \right\}.
\end{equation*}
We let $c_i = (a_i)_{i=1}^5$ with $a_i = 0$ and $a_j = 1$ for $i \neq j$. Notice that the $c_i$ generate $C_2^4$. Also note that if $I \subseteq \{1,2,3,4,5\}$, then $\prod_{i \in I} c_i = \prod_{i \not\in I} c_i$ in $C_2^4$.

\subsection{The 16 lines} A del Pezzo surface $X$ of degree $d$ is a smooth projective surface with an ample anticanonical divisor $-K_X$ such that $K_X^2 = d$. Let $X$ be a del Pezzo surface of degree 4 over $k$ (i.e. a quartic del Pezzo surface). Then $\overline{X}$ is the blowup $\pi: \overline{X} \to \mathbb{P}^2_{\bar{k}}$ of five points $P_1,...,P_5$ in general position on $\mathbb{P}^2_{\bar{k}}$. As a result, $\Pic \overline{X} \cong \mathbb{Z}^6$ and $\Pic \overline{X}$ is generated by $\{H, E_1,...,E_5\}$ where $E_i = \pi^{-1}(P_i)$ and $H$ is the pullback by $\pi$ of a line in $\mathbb{P}^2_{\bar{k}}$. The canonical class of $\Pic \overline{X}$ is $K_X = -3H + \sum_{i=1}^5 E_i$. The intersection pairing on $\Pic \overline{X}$ is determined by the following relations:
\begin{equation*}
    E_i \cdot E_j = -\delta_{ij}; \quad H^2 = 1; \quad H \cdot E_i = 0.
\end{equation*}

\begin{definition}
    An effective divisor $D$ on $X$ is \textbf{exceptional} if $D^2 = D \cdot K_X = -1$.
\end{definition}

\noindent Any quartic del Pezzo surface has exactly 16 exceptional divisors. They are
\begin{itemize}
    \item[(i)] $E_i = \pi^{-1}(P_i)$ for $1 \leq i \leq 5$.

    \item[(ii)] $L_{ij} = \wt{\pi^{-1}}(\mathcal{L}_{ij}) = H - E_i - E_j$; the strict transform of the line $\mathcal{L}_{ij}$ through $P_i$ and $P_j$.

    \item[(iii)] $C = \wt{\pi^{-1}}(\mathcal{C}) = 2H - \sum_{i=1}^5 E_i$; the strict transform of the unique conic $\mathcal{C}$ through $P_1,...,P_5$.
\end{itemize}

\noindent The 16 exceptional divisors correspond to the 16 lines defined over $\bar{k}$ on $X$ under the anticanonical embedding into $\mathbb{P}^4_k$. Following the exposition of Section 8.1 in \cite{Blanc}, we relate $W(\mathsf{D}_5)$ to the action of $\Aut(\overline{X})$ on the exceptional pairs of $\Pic \overline{X}$.

\begin{definition}
    A pair of exceptional divisors $\{C,D\}$ on $X$ is an \textbf{exceptional pair} if $C^2 = D^2 = 0$ and $C+D = -K_X$.
\end{definition}

\begin{lem}[Lemma 8.1.2 of \cite{Blanc}]
    Let $X$ be a quartic del Pezzo surface. Then $X$ has exactly five exceptional pairs given by $\{H - E_i, -K_X - H + E_i\} \text{ for } 1 \leq i \leq 5$.
\end{lem}

\noindent Geometrically, each pair corresponds to a line through one of the blown-up points and the conic through the other four points. The action of $\Aut(\overline{X})$ on the exceptional pairs induces a homomorphism $\rho: \Aut(\overline{X}) \to S_5$. We have an exact sequence of $\Gamma$-groups
\begin{equation}
\label{exactsequence}
    1 \to A \to \Aut(\overline{X}) \xrightarrow{\rho} H \to 1
\end{equation}
where $H$ is the image of $\rho$ in $S_5$. Note that $A$ and $H$ vary depending on the choice of $X$. We define a group homomorphism $\gamma: A \to C_2^5$ by $\gamma(g) = (a_i)_{i=1}^5$ where $a_i = 1$ if $g$ swaps $H - E_i$ with $-K_X - H + E_i$ and $a_i = 0$ if $g$ fixes both.

\begin{prop}[Prop. 8.1.3 of \cite{Blanc}]
\label{prop:Aut(X)_structure}
$A$ and $H$ have the following structure as abstract groups:
\begin{itemize}
    \item[(i)] $\gamma$ induces an isomorphism $A \cong \{(a_1,a_2,a_3,a_4,a_5) \in C_2^5 | \sum_{i=1}^5 a_i = 0\} \cong C_2^4$.

    \item[(ii)] $H$ is identified with the subgroup of $\PGL_3(\bar{k})$ stabilizing $\{P_1,...,P_5\}$.

    \item[(iii)] $\Aut(\overline{X}) \cong A \rtimes H$ where $H$ acts on $A$ by permuting the coordinates.
\end{itemize}
\end{prop}

\begin{remark}
    The unique conic through $\{P_1,...,P_5\}$ allows us to identify $\{P_1,...,P_5\}$ with five points of $\mathbb{P}^1_{\bar{k}}$. Then $H$ can be identified with the subgroup of $\PGL_2(\bar{k})$ stabilizing the five points of $\mathbb{P}^1_{\bar{k}}$.
\end{remark}

\begin{remark}
    Proposition \ref{prop:Aut(X)_structure} allows us to identify $\Aut(\overline{X})$ with a subgroup of \newline $C_2^4 \rtimes S_5 \cong W(\mathsf{D}_5)$. 
\end{remark}

\noindent We can explicitly describe the action of $A$ on the 16 exceptional divisors. 

\begin{lem}
\label{lemma:Aaction}
Let $c_i = (a_i)_{i=1}^5 \in A$ with $a_i = 0$ and $a_j = 1$ for $i \neq j$. Then
\begin{itemize}
    \item $c_i(E_i) = C$;
    \item $c_i(E_j) = L_{ij}$ when $j \neq i$;
    \item $c_i(L_{jk}) = L_{\{i,j,k\}^c}$ when $i,j,k$ are distinct in $\{1,2,3,4,5\}$.
\end{itemize}
Since $c_i^2 = \id$, this determines where $c_i$ sends each exceptional divisor.
\end{lem}
\begin{proof}
    Let $I = \{1,2,3,4,5\}$. We let $\Omega_j = H-E_j$ and $\overline{\Omega_j} = -K_{X} - H + E_j$ for any $j \in I$ so that the exceptional pairs of $X$ are $\{\Omega_j, \overline{\Omega_j}\}$. We write the exceptional divisors as linear combinations of divisors within the exceptional pairs.
    \begin{equation*}
        E_i = \frac{3}{2} \overline{\Omega_i} + \Omega_i - \frac{1}{2}\sum_{k \neq i} \overline{\Omega_k}, \quad L_{ij} = \frac{3}{2} \Omega_i + \overline{\Omega_i} - \frac{1}{2} \overline{\Omega_j} - \frac{1}{2}\sum_{k \neq i,j} \Omega_k
    \end{equation*}
    for all $i \in I$ and distinct $i,j \in I$. Then
    \begin{align*}
        c_i(E_i) &= \frac{3}{2}\overline{\Omega_i} + \Omega_i - \frac{1}{2} \sum_{k \neq i} \Omega_k = C \\
        c_i(E_j) &= \frac{3}{2}\Omega_j + \overline{\Omega_j} - \frac{1}{2}\overline{\Omega_i} - \frac{1}{2}\sum_{k \neq i,j} \Omega_k = L_{ij} \\
        c_i(L_{jk}) &= \frac{3}{2}\overline{\Omega_j} + \Omega_j - \frac{1}{2}\Omega_k - \frac{1}{2}\Omega_i - \frac{1}{2}\sum_{r \neq i,j,k} \overline{\Omega_r} = L_{\{i,j,k\}^c}
    \end{align*}
    Since $c_i$ has order two, we have also determined where $c_i$ sends $C, L_{ij}$, and $L_{\{i,j,k\}^c}$.
\end{proof}

\noindent Using the previous lemma, one can check that the action of $A$ on the exceptional divisors is transitive. Therefore, the exceptional divisors form an $A$-torsor.

A sequence of birational morphisms
\begin{equation*}
    \pi: \overline{X} = X_5 \xrightarrow{\pi_5} X_4 \xrightarrow{\pi_4} ... \xrightarrow{\pi_2} X_1 \xrightarrow{\pi_1} \mathbb{P}^2
\end{equation*}
where each $\pi_i: X_i \to X_{i-1}$ is the blow-up of a point of $X_{i-1}$ determines a \textit{geometric basis} $\{H,E_1,...,E_5\}$ for $\Pic \overline{X}$. After choosing a geometric basis, the action of $\Aut(\overline{X})$ on the exceptional pairs induces an injective group homomorphism $\Aut(\overline{X}) \to W(\mathsf{D}_5)$. A different choice of geometric basis will produce a conjugate subgroup in $W(\mathsf{D}_5)$ (see Lemma 6.1 of \cite{DI09}). If $G$ is a subgroup of the image of $\Aut(X)$ in $W(\mathsf{D}_5)$, then we say $G$ \textit{acts by automorphisms on $X$}. Of course, this notion is defined up to conjugacy in $W(\mathsf{D}_5)$. After specifying a field $k$, we can ask which subgroups of $W(\mathsf{D}_5)$ --- up to conjugacy --- act by automorphisms on some quartic del Pezzo surface over $k$.

\subsection{The five points of \texorpdfstring{$\mathbb{P}^1$}{P1}}

Let $X$ be a del Pezzo surface of degree 4 over $k$. Under the anticanonical embedding, $X$ is a complete intersection of quadrics $V(q_1) \cap V(q_2)$ in $\mathbb{P}^4_k$. Let $\mathcal{P} = \mu q_1 + \lambda q_2$ be the pencil spanned by the two quadrics, and let $\Delta$ be the discriminant of $\mathcal{P}$. Then $\Delta$ is a reduced degree five binary form over $k$, and the $\bar{k}$-points of $\Delta$ correspond to singular surfaces in the pencil. We define $S_X$ to be the subscheme of $\mathbb{P}^1_k$ corresponding to $\Delta$. Over $\bar{k}$, the quadratic forms $q_1$ and $q_2$ can be simultaneously diagonalized and written in the form
\begin{equation*}
    q_1 = \sum_{i=0}^4 x_i^2 \quad \text{and} \quad q_2 = \sum_{i=0}^4 \alpha_ix_i^2
\end{equation*}
where the $\alpha_i$ are distinct (see Theorem 8.6.2 of \cite{D12}). In this form, mapping $x_i$ to $-x_i$ induces an automorphism of the surface. These diagonal automorphisms form a subgroup of $\Aut(\overline{X})$ isomorphic to $C_2^4$ that is identified with $A$. The group $\Aut(\overline{X})$ also acts on the coordinates $(\mu:\lambda)$ of the pencil of quadrics and preserves the form $\Delta$, so we can write $\Aut(\overline{X})$ as a semi-direct product
\begin{equation*}
    \Aut(\overline{X}) \cong C_2^4 \rtimes H
\end{equation*}
where $H$ is the finite subgroup of $\PGL_2(\bar{k})$ stabilizing the five $\bar{k}$-points of $S_X$. This recovers the exact sequence (\ref{exactsequence}) from the previous section. If $G$ is a subgroup of $\PGL_2(\bar{k})$ stabilizing five points on $\mathbb{P}^1_{\bar{k}}$, then $G$ is trivial or isomorphic to $C_2$, $C_3$, $C_4$, $C_5$, $S_3$, or $D_5$. Moreover, $\Aut(\overline{X})$ is one of the following five groups:
\begin{equation*}
    C_2^4, \quad C_2^4 \rtimes C_2, \quad C_2^4 \rtimes C_4, \quad C_2^4 \rtimes S_3, \quad C_2^4 \rtimes D_5.
\end{equation*}
See Section 8.6.4 of \cite{D12} for details. 

\begin{remark}
    Note that $H$ --- viewed as a subgroup of $\PGL_2(\bar{k})$ --- permutes the five points of $\mathbb{P}^1_{\bar{k}}$ in the same way it permutes the exceptional pairs of $\Pic \overline{X}$. Also, $\Gamma$ permutes $\{c_1,...,c_5\}$ in $A$ in the same way $\Gamma$ permutes the five points of $\mathbb{P}^1_{\bar{k}}$.
\end{remark}

\noindent We denote $\langle c_1,c_2,c_3,c_4,c_5 \rangle$ in $W(\mathsf{D}_5)$ as $C_2^4$. Using the binary forms in Section 8.6.4 of \cite{D12}, we can easily list the images of the possible $\Aut(\overline{X})$ in $W(\mathsf{D}_5)$. We provide names for the subgroups up to conjugacy in $W(\mathsf{D}_5)$ that will appear frequently:
\begin{table}[h!]
\caption{Names for conjugacy classes. \label{table:classes}}
    \begin{tabular}{|c|c|}
        \hline
        \textbf{Name} & \textbf{Representative of Class} \\
        \hline
        $C_2^4 \rtimes C_2$ & $C_2^4 \rtimes \langle (12)(45) \rangle$ \\
        \hline
        $C_2^4 \rtimes C_4$ & $C_2^4 \rtimes \langle (1425) \rangle$ \\
        \hline
        $C_2^4 \rtimes S_3$ & $C_2^4 \rtimes \langle (12)(45), (123) \rangle$ \\
        \hline
        $C_2^4 \rtimes D_5$ & $C_2^4 \rtimes \langle (12)(45), (15342) \rangle$ \\
        \hline
        $C_2^3 \rtimes S_3$ & $\langle c_1,c_2,c_3 \rangle \rtimes \langle (12)(45), (123) \rangle$ \\
        \hline
        $C_2^3 \cdot S_3$ & $\langle c_1,c_2,c_3,(123), c_4(12)(45) \rangle$ \\
        \hline
        $C_2 \cdot S_3$ & $\langle c_4c_5, (123), c_4(12)(45) \rangle$ \\
        \hline
    \end{tabular}
\end{table}

\section{Quasi-Split Surfaces} 

In \cite{Skorobogatov}, all del Pezzo surfaces of degree 4 over $k$ are identified with twists of what are termed ``quasi-split" surfaces. We will first determine the automorphisms of the quasi-split surfaces and then analyze the automorphisms of the twisted surfaces in the next section.

\begin{definition}
\label{def:quasi-split}
    A del Pezzo surface $X$ of degree 4 over $k$ is called \textbf{split} if all 16 lines on $\overline{X}$ are defined over $k$. A del Pezzo surface $X$ of degree 4 is called \textbf{quasi-split} if at least one of the 16 lines is defined over $k$. 
\end{definition}

\noindent Equivalently, a surface $X$ is quasi-split if it is the blow-up of $\mathbb{P}^2_k$ in a Galois-stable set of five $\bar{k}$-points in general position. In the previous section we constructed a reduced degree five subscheme $S_X$ of $\mathbb{P}^1_k$ corresponding to each del Pezzo surface $X$ of degree 4. Conversely, given a reduced degree five subscheme $S$ of $\mathbb{P}^1_k$, the blow-up of $\mathbb{P}^2_k$ in the image of $S$ under the Veronese embedding $\nu: \mathbb{P}^1 \xhookrightarrow{} \mathbb{P}^2$ produces a quasi-split surface $X_0$. The following lemma is due to Skorobogatov \cite{Skorobogatov}:

\begin{lem}
    Any quasi-split del Pezzo surface $X$ of degree 4 is isomorphic to the blow-up of $\mathbb{P}^2_k$ in the image of $S_X$ under the Veronese embedding $\mathbb{P}^1_k \xhookrightarrow{} \mathbb{P}^2_k$.
\end{lem}

\noindent In other words, we have a bijection between reduced degree five subschemes of $\mathbb{P}^1_k$ up to projective equivalence and quasi-split del Pezzo surfaces of degree 4 over $k$ up to isomorphism.

\begin{lem}
\label{Lemma:Moving}
    Let $S$ be a reduced degree five subscheme of $\mathbb{P}^1_k$, and let $X_0$ be the quasi-split surface obtained as the blow-up of $\mathbb{P}^2_k$ in the image of $S$ under the Veronese embedding. Let $P_1,...,P_5$ be the five $\bar{k}$-points of $S$.
    \begin{itemize}
        \item[(i)] The exact sequence \eqref{exactsequence} splits as a sequence of $\Gamma$-groups.

        \item[(ii)] $\Aut(X_0) \cong C_2^{n-1} \rtimes H^{\Gamma}$ where $n$ is the number of $\Gamma$-orbits of $\{P_1,...,P_5\}$. 

        \item[(iii)] $A$ and the 16 exceptional divisors of $X_0$ are isomorphic as $\Gamma$-sets.
    \end{itemize}
\end{lem}

\begin{proof} 
(i) We identify $H$ with the subgroup of $\PGL_2(\bar{k})$ stabilizing $P_1,...,P_5$. There is a $\Gamma$-equivariant map $\phi: H \xhookrightarrow{} \PGL_3(\bar{k})$ induced by the Veronese embedding where $\phi(H)$ stabilizes the image of $\{P_1,...,P_5\}$ in $\mathbb{P}_{\bar{k}}^2$. Let $\pi: \overline{X_0} \to \mathbb{P}^2_{\bar{k}}$ be the blow-up map. Since $\{P_1,...,P_5\}$ is $\Gamma$-stable, the rational maps $\pi$ and $\pi^{-1}$ are $\Gamma$-equivariant. We define the splitting to be $\psi: H \to \Aut(\overline{X_0})$ given by $h \mapsto \pi^{-1}\phi(h)\pi$. It is clear that this map is $\Gamma$-equivariant and that $\rho \circ \psi = \id_H$. 

\medskip

\noindent (ii) From (i) we have $\Aut(X_0) \cong A^{\Gamma} \rtimes H^{\Gamma}$. Recall that $A$ is generated by $\{c_1,...,c_5\}$ where $c_i = (a_j)_{j=1}^5$ with $a_j = 0$ when $i = j$ and $a_j = 1$ when $i \neq j$. Moreover, $\Gamma$ permutes $\{c_1,...,c_5\}$ in the same way it permutes $\{P_1,...,P_5\}$. Each element of $A$ can be written as a product of the $c_i$ in exactly two ways. Let $S$ be a subset of $\{c_1,...,c_5\}$. If $g = \prod_{c \in S} c$ then $g = \prod_{c \not\in S} c$. Now $g \in A$ is in $A^{\Gamma}$ if and only if $g = \prod_{c \in S} c$ where $S$ is $\Gamma$-stable. Of course, $S$ is $\Gamma$-stable if and only if $S$ is the union of $\Gamma$-orbits in $\{c_1,...,c_5\}$. Let $n$ be the number of $\Gamma$-orbits of $\{c_1,...,c_5\}$. Then there are $2^n$ $\Gamma$-stable subsets of $\{c_1,...,c_5\}$. Each $g \in A^{\Gamma}$ corresponds to exactly two $\Gamma$-stable subsets of $\{c_1,...,c_5\}$, so $A^{\Gamma} \cong C_2^{n-1}$.

\medskip

\noindent (iii) We described the action of $A$ on the exceptional divisors in Lemma \ref{lemma:Aaction}. Since $X_0$ is quasi-split, $C$ is fixed by the $\Gamma$-action on $\Pic \overline{X_0}$, and the exceptional divisors form a trivial $A$-torsor. Consequently, $A$ is isomorphic to the 16 exceptional divisors via the map $g \mapsto g(C)$.
\end{proof}

\begin{remark}
\label{Remark:A}
    In part (ii) of Lemma \ref{Lemma:Moving}, we did not need to assume that $X$ was quasi-split to show that $A^{\Gamma} \cong C_2^{n-1}$ where $n$ is the number of $\Gamma$-orbits in $\{P_1,...,P_5\}$. This fact holds for non-quasi-split surfaces as well.
\end{remark}

As a result of the previous lemma, to classify the automorphisms of quasi-split surfaces it suffices to analyze reduced degree five subschemes $S$ of $\mathbb{P}^1_k$, the possible $\Gamma$-actions on $S$, and the subgroup of $\PGL_2(k)$ leaving $S$ invariant. We produce examples of quasi-split surfaces over various fields and calculate $\Aut(X_0)$. We will later prove that each of the constructions produces a group that is maximal in $\mathcal{M}_k$, the collection of subgroups in $W(\mathsf{D}_5)$ acting on a quartic del Pezzo surface over $k$, for some field $k$. We first label the following elements of $\PGL_2(\bar{k})$:
\begin{equation*}
    h_2 = \begin{bmatrix} 0 & 1 \\ 1 & 0 \end{bmatrix}; \quad h_3 = \begin{bmatrix} 0 & -1 \\ 1 & -1 \end{bmatrix}; \quad h_4 = \begin{bmatrix} 1 & 0 \\ 0 & i \end{bmatrix}; \quad h_5 = \begin{bmatrix} 1 & 1 \\ -1 & \phi \end{bmatrix}
\end{equation*}
where $\phi = \dfrac{1 + \sqrt{5}}{2}$. Notice that $h_n$ is an element of order $n$ in $\PGL_2(\bar{k})$. In each of the following examples, $X_0$ will denote the quasi-split surface corresponding to the subscheme $S$ of $\mathbb{P}^1_k$.

\begin{example}
\label{example:C2}
    Let $k$ be an arbitrary field. Let $H = \langle h_2 \rangle$. Then $H$ stabilizes any set of five points in $\mathbb{P}^1_{\bar{k}}$ of the form
    \begin{equation*}
        S = \{(a:b), (b:a), (1:1), (c:d), (d:c)\}
    \end{equation*}
    We can choose $a,b,c,d \in k$ so that the $\Gamma$-action on $S$ is trivial. For a general choice of $(a:b)$ and $(c:d)$, the group $H$ is the full subgroup of $\PGL_2(k)$ stabilizing $S$. In that case, the image of $\Aut(X_0)$ in $W(\mathsf{D}_5)$ is conjugate to $C_2^4 \rtimes \langle (12)(45) \rangle$.
\end{example}

\begin{example}
\label{example:S3}
    Let $H = \langle h_2,h_3 \rangle \cong S_3$. Then $H$ stabilizes
    \begin{equation*}
        S = \{(1:0), (0:1), (1:1), (-\epsilon_3:1), (1:-\epsilon_3)\}
    \end{equation*}
    If $\epsilon_3 \in k$, then the $\Gamma$-action fixes each point of $S$ and $\Aut(X_0) = C_2^4 \rtimes \langle (12)(45), (123) \rangle$. If $\epsilon_3 \not\in k$, then the $\Gamma$-action fixes the first three points and swaps $(-\epsilon_3 :1)$ with $(1:-\epsilon_3)$, so $\Aut(X_0) = \langle c_1,c_2,c_3 \rangle \rtimes \langle (12)(45), (123) \rangle$.
\end{example}

\begin{example}
\label{example:C4}
    Suppose $i \in k$. Let $H = \langle h_4 \rangle \cong C_4$. Then $H$ stabilizes
    \begin{equation*}
        S = \{ (a:b), (a:-b), (1:0), (a:ib), (a:-ib)\}
    \end{equation*}
    for any $a,b \in k$. Since $i \in k$, the $\Gamma$-action on $S$ is trivial and $\Aut(X_0) = C_2^4 \rtimes \langle (1425) \rangle$.
\end{example}

\begin{example}
\label{example:D5}
    Suppose $\sqrt{5} \in k$. Let $H = \langle h_2, h_5 \rangle \cong D_5$. Then $H$ stabilizes
    \begin{equation*}
        S = \{ (\phi:1), (1:\phi), (1:-1), (0:1), (1:0)\}
    \end{equation*}
    Since $\phi \in k$, the $\Gamma$-action on $S$ is trivial and $\Aut(X_0) = C_2^4 \rtimes \langle (12)(45), (15342) \rangle$.
\end{example}

\noindent The previous examples construct groups in $\mathcal{M}_k$ for various $k$. We would also like to determine necessary conditions on $k$ for various groups in $W(\mathsf{D}_5)$ to be in $\mathcal{M}_k$. To obtain some necessary conditions, we will use the following fact taken from Proposition 1.1 and Theorem 4.2 of \cite{Beauville}.

\begin{prop}
\label{Prop:Beauville}
\,
\begin{itemize}
    \item[(i)] $\PGL_2(k)$ contains $C_n$ and $D_n$ if and only if $k$ contains $\epsilon_n + \epsilon_n^{-1}$.

    \item[(ii)] Let $n > 2$. If $\PGL_2(k)$ contains a subgroup isomorphic to $C_n$ then it is unique up to conjugacy.
\end{itemize}
\end{prop}

\begin{lem}
\label{Lemma:Pre-quasi}
    Let $k$ be a field of characteristic zero. Let $X$ be a del Pezzo surface of degree 4 over $k$ with an exact sequence
    \begin{equation*}
        1 \to A \to \Aut(\overline{X}) \xrightarrow{\rho} H \to 1
    \end{equation*}
    of $\Gamma$-groups.
    \begin{itemize}
        \item[(i)] If $H^{\Gamma}$ contains an element of order 4, then $i \in k$.
        \item[(ii)] If $H^{\Gamma}$ contains an element of order 5, then $\sqrt{5} \in k$.
        \item[(iii)] If $H^{\Gamma}$ contains an element of order 3 and $A^{\Gamma} \cong C_2^4$, then $\epsilon_3 \in k$.
        \item[(iv)] If $H^{\Gamma}$ contains an element acting on the five points of $S_X$ as $(123)$, then $A^{\Gamma}$ is a subgroup of $\langle c_1,c_2,c_3 \rangle$.
    \end{itemize}
\end{lem}

\begin{proof}
We identify $H$ with the subgroup of $\PGL_2(\bar{k})$ stabilizing the five $\bar{k}$-points $\{P_1,...,P_5\}$ of $S_X$. Note that $H^{\Gamma}$ is a subgroup of $\PGL_2(k)$.

\medskip

\noindent (i) Suppose $i \not\in k$. By Proposition \ref{Prop:Beauville}, there is a unique subgroup of order 4 in $\PGL_2(k)$ up to conjugacy. We pick a representative to be generated by
\begin{equation*}
    h_4' = \begin{bmatrix} 1 & 1 \\ -1 & 1 \end{bmatrix} \in \PGL_2(k).
\end{equation*}
Notice that $h_4'$ is conjugate to $h_4$ when $i \in k$. The fixed points of $h_4'$ are $(1:-i)$ and $(-i:1)$. Since $i \not\in k$, $\{(1:-i), (-i:1)\}$ forms a $\Gamma$-orbit of $\mathbb{P}^1_{\bar{k}}$. The fixed points of any element conjugate to $h_4'$ in $\PGL_2(k)$ are also not $\Gamma$-fixed. Suppose we have five distinct $\bar{k}$-points stabilized by $\Gamma$ and $h_4'$. Then exactly one of the points is fixed by $h_4'$. But for the set to be $\Gamma$-stable, it must contain the second fixed point of $h_4'$, a contradiction. So if $H^{\Gamma} \cong C_4$ then $k$ must contain $i$.

\medskip

\noindent (ii) By Proposition \ref{Prop:Beauville}, $\PGL_2(k)$ contains an element of order 5 if and only if $\epsilon_5 + \epsilon_5^{-1} \in k$. Of course, $\epsilon_5 + \epsilon_5^{-1} \in k$ if and only if $\sqrt{5} \in k$.

\medskip

\noindent (iii) Suppose $\epsilon_3 \not\in k$. By Proposition \ref{Prop:Beauville}, there is a unique subgroup of order 3 in $\PGL_2(k)$ up to conjugacy. Without loss of generality, we may assume the group is generated by $h_3$.
The fixed points of $h_3$ are $(1:-\epsilon_3)$ and $(-\epsilon_3:1)$. Since $\epsilon_3 \not\in k$, we know $\{(1:-\epsilon_3), (-\epsilon_3:1)\}$ forms a $\Gamma$-orbit of $\mathbb{P}^1_{\bar{k}}$. Any $\Gamma$-stable set of five $\bar{k}$-points must contain both fixed points, and the fixed points of any conjugate element of $\PGL_2(k)$ will also not be $\Gamma$-fixed. Therefore, there are at most four $\Gamma$-orbits in $\{P_1,...,P_5\}$. By Lemma \ref{Lemma:Moving} and Remark \ref{Remark:A}, we conclude that $A^{\Gamma} \not\cong C_2^4$.

\medskip

\noindent (iv) By (iii), if $H^{\Gamma}$ contains an element of order 3 and $\epsilon_3 \not\in k$, then we can assume $\{P_4,P_5\}$ forms a $\Gamma$-orbit in $\{P_1,...,P_5\}$. Therefore, $c_4$ and $c_5$ are not in $A^{\Gamma}$ and $A^{\Gamma} \subseteq \langle c_1,c_2,c_3 \rangle$.
\end{proof}

We let $\mathcal{M}_k^{qs}$ denote the collection of subgroups, up to conjugacy, in $W(\mathsf{D}_5)$ that act by automorphisms on a quasi-split del Pezzo surface of degree 4 over $k$. At this point we are able to determine the structure of $\mathcal{M}_k^{qs}$ for any field $k$ of characteristic zero.

\begin{thm}
    \label{Theorem:Quasi-split}
    Let $k$ be a field of characteristic zero. If $G$ is a maximal group in $\mathcal{M}_k^{qs}$, then $G$ is one of the groups in the table. Each group appears in $\mathcal{M}_k^{qs}$ if and only if $k$ satisfies the condition in the second column.
    \begin{center}
    \begin{tabular}{|c|c|}
        \hline
        \textbf{Group} & \textbf{Condition on } $k$ \\
        \hline
        $C_2^4 \rtimes C_4$ & $i \in k$ \\
        \hline
        $C_2^4 \rtimes S_3$ & $\epsilon_3 \in k$ \\
        \hline
        $C_2^4 \rtimes D_5$ & $\sqrt{5} \in k$ \\
        \hline
        $C_2^4 \rtimes C_2$ & none \\
        \hline
        $C_2^3 \rtimes S_3$ & none \\
        \hline
    \end{tabular}
    \end{center}
\end{thm}

\begin{proof}
We first show that each group in the table appears in $\mathcal{M}_k^{qs}$ if and only if the condition on $k$ is satisfied. The following statements are easy consequences of Lemma \ref{Lemma:Pre-quasi}:
\begin{itemize}
    \item[(i)] If $C_2^4 \rtimes C_4 \in \mathcal{M}_k$, then $i \in k$.
    \item[(ii)] If $C_2^4 \rtimes S_3 \in \mathcal{M}_k$, then $\epsilon_3 \in k$.
    \item[(iii)] If $C_2^4 \rtimes D_5 \in \mathcal{M}_k$, then $\sqrt{5} \in k$.
\end{itemize}
For the reverse direction, Examples \ref{example:C4}, \ref{example:S3}, \ref{example:D5}, and \ref{example:C2} produce quasi-split surfaces with $\Aut(X_0)$ equal to the groups in the left column under the appropriate conditions.

\medskip

\noindent For each $k$, we use the constructions in the Examples and Lemma \ref{Lemma:Pre-quasi} to determine the maximal groups in $\mathcal{M}_k^{qs}$. Recall that the maximal groups in $\mathcal{M}_{\bar{k}}$ are $C_2^4 \rtimes C_4$, $C_2^4 \rtimes S_3$, and $C_2^4 \rtimes D_5$. We break into cases:
\begin{itemize}
    \item[Case 1:] Suppose $i,\epsilon_3,\sqrt{5} \in k$. The maximal groups are $C_2^4 \rtimes C_4$, $C_2^4 \rtimes S_3$, and $C_2^4 \rtimes D_5$. 

    \item[Case 2:] Suppose $\epsilon_3,\sqrt{5} \in k$ and $i \not\in k$. Then $C_2^4 \rtimes S_3$ and $C_2^4 \rtimes D_5$ are maximal. Any other maximal group would be contained in $C_2^4 \rtimes C_4$. By Lemma \ref{Lemma:Pre-quasi}, such a group would need to be contained in $C_2^4 \rtimes C_2$, which is already in $C_2^4 \rtimes S_3$.

    \item[Case 3:] Suppose $i,\sqrt{5} \in k$ and $\epsilon_3 \not\in k$. Then $C_2^4 \rtimes C_4$ and $C_2^4 \rtimes D_5$ are maximal. Any other maximal group would be contained in $C_2^4 \rtimes S_3$. By Lemma \ref{Lemma:Pre-quasi}, such a group is contained in $\langle c_1,c_2,c_3 \rangle \rtimes \langle (12)(45), (123) \rangle$. So $C_2^3 \rtimes S_3$ is the only other maximal group.

    \item[Case 4:] Suppose $i, \epsilon_3 \in k$ and $\sqrt{5} \not\in k$. Then $C_2^4 \rtimes C_4$ and $C_2^4 \rtimes S_3$ are maximal. Any other maximal group would be contained in $C_2^4 \rtimes D_5$. By Lemma \ref{Lemma:Pre-quasi}, such a group is contained in $C_2^4 \rtimes C_2$ which is already contained in $C_2^4 \rtimes S_3$.
    
    \item[Case 5:] Suppose $i \in k$ and $\epsilon_3, \sqrt{5} \not\in k$. Then $C_2^4 \rtimes C_4$ and $C_2^3 \rtimes S_3$ are the maximal groups.

    \item[Case 6:] Suppose $\epsilon_3 \in k$ and $i,\sqrt{5} \not\in k$. Then $C_2^4 \rtimes S_3$ is the only maximal group.

    \item[Case 7:] Suppose $\sqrt{5} \in k$ and $i,\epsilon_3 \not\in k$. Then $C_2^4 \rtimes D_5$ and $C_2^3 \rtimes S_3$ are the maximal groups.

    \item[Case 8:] Suppose $i,\epsilon_3,\sqrt{5} \not\in k$. Then $C_2^4 \rtimes C_2$ and $C_2^3 \rtimes S_3$ are the maximal groups.
\end{itemize}

\noindent Therefore, for any $k$ the maximal groups in $\mathcal{M}_k^{qs}$ are a subset of the groups listed in the table.
\end{proof}

\begin{remark}
    After fixing a field $k$ of characteristic zero, $\mathcal{M}_k^{qs}$ is precisely the downward closure of the groups in the table that are in $\mathcal{M}_k^{qs}$. 
\end{remark}

\section{Twists of Quasi-Split Surfaces} 

At this point, we would like to extend the classification to \textit{any} del Pezzo surface of degree 4. Let $X$ be a del Pezzo surface of degree 4 over $k$. Recall that $A$ is the kernel of the map from $\Aut(\overline{X})$ to $S_5$ induced by the action of $\Aut(\overline{X})$ on the exceptional pairs of $\Pic \overline{X}$. Since $X$ is a projective variety with an action of $A$, we can twist $X$ by an element of $H^1(\bar{k}/k, A)$ (see I.5 and III.1 of \cite{Serre} for details). Recall that the classes of $H^1(\bar{k}/k,A)$ are in bijection with isomorphism classes of $A$-torsors. We showed in Lemma \ref{lemma:Aaction} that the 16 exceptional divisors of $X$ form an $A$-torsor. If $\tau$ is an $A$-torsor, the twist of $X$ by $\tau$ --- denoted ${}^\tau X$ --- is defined to be the quotient of $\tau \times_k X$ by the diagonal action of $A$. Note that the surface $X$ is quasi-split precisely when the torsor of lines has a $k$-point. We can also describe the twisting operation with cocycles. For $(a_{\gamma}) \in Z^1(\bar{k}/k, A)$, we can define a new $\Gamma$ action on $\overline{X}$ by combining $a_{\gamma}$ with the automorphisms $1 \otimes \gamma$ of $\overline{X}$. The twist of $X$ by $(a_{\gamma})$ --- denoted ${}^a X$ --- is the quotient of $\overline{X}$ by the new $\Gamma$-action. The twist will be a $\bar{k}/k$-form of $X$. The following result is due to A. Skorobogatov \cite{Skorobogatov}:

\begin{prop}
\label{Prop:Skorobogatov}
    Let $X$ be a del Pezzo surface of degree 4, and let $S_X$ be the attached reduced degree five subscheme of $\mathbb{P}^1_k$. Let $X_0$ be the blow-up of $\mathbb{P}^2_k$ in the image of $S_X$ under the Veronese embedding. Then $X_0$ is
    \begin{itemize}
        \item[(a)] the unique (up to isomorphism) quasi-split twist of $X$ by a $k$-torsor under $A$;

        \item[(b)] the unique (up to isomorphism) quasi-split del Pezzo surface of degree 4 such that $S_X$ and $S_{X_0}$ are projectively equivalent as subschemes of $\mathbb{P}^1_k$.
    \end{itemize}
\end{prop}

\noindent It follows that the isomorphism classes of del Pezzo surfaces of degree 4 over $k$ are in natural bijection with pairs $(S,[a])$ where $S$ is a reduced degree five subscheme of $\mathbb{P}^1_k$ --- up to projective equivalence --- and $[a] \in H^1(\bar{k}/k,A)$. 

The $\Gamma$ action on $\overline{X}$ induces a $\Gamma$ action on $\Aut(\overline{X})$ defined by ${}^{\gamma} g = \gamma g \gamma^{-1}$ for $\gamma \in \Gamma$ and $g \in \Aut(\overline{X})$. If we anticanonically embed $\overline{X}$ in $\mathbb{P}^4_{\bar{k}}$ and identify $\Aut(\overline{X})$ with elements of $\PGL_5(\bar{k})$, then the induced action corresponds to an entrywise action of $\Gamma$ on the representing matrices. After twisting $X$ by a cocycle $(a_\gamma) \in Z^1(\bar{k}/k, A)$, the induced $\Gamma$ action on $\Aut(\overline{{}^a X})$ is defined by ${}^{\gamma'}g = a_{\gamma} \gamma g \gamma^{-1} a_{\gamma}^{-1}$ for all $\gamma \in \Gamma$ and $g \in \Aut(\overline{{}^a X})$. 

\medskip

\noindent Let $G$ be a subgroup of $W(\mathsf{D}_5) \cong C_2^4 \rtimes S_5$ acting by automorphisms on some del Pezzo surface of degree 4 over $k$. We have an exact sequence
\begin{equation*}
    1 \to G' \to G \to G'' \to 1
\end{equation*}
where $G''$ is the image of the induced map from $G$ to $S_5$. 

\begin{lem}
\label{Lemma:ReducetoS3}
    Let $k$ be a field of characteristic zero. Let $G$ be a subgroup of $W(\mathsf{D}_5)$ acting by automorphisms on some del Pezzo surface of degree 4 over $k$. If $G''$ is not isomorphic to $S_3$, then $G$ acts by automorphisms on a quasi-split del Pezzo surface of degree 4 over $k$. 
\end{lem}

\begin{proof}
Since $G$ acts by automorphisms on some quartic del Pezzo surface, $G''$ is a subgroup of $C_4$, $S_3$, or $D_5$. Moreover, if $G$ acts by automorphisms on a quartic del Pezzo surface $X$ with
\begin{equation*}
    1 \to A \to \Aut(\overline{X}) \to H \to 1
\end{equation*}
the resulting exact sequence of $\Gamma$-groups, then $G''$ is a subgroup of $H^{\Gamma}$. We break into cases:

\begin{itemize}
    \item[Case 1:] Suppose $G''$ is trivial. Then $G$ is a subgroup of $C_2^4$ in $W(\mathsf{D}_5)$, which acts by automorphisms on a quasi-split quartic del Pezzo surface over any $k$ by Theorem \ref{Theorem:Quasi-split}.

    \item[Case 2:] Suppose $G'' \cong C_2$. Then $G$ is a subgroup of $C_2^4 \rtimes \langle (12)(45) \rangle$ up to conjugacy, which acts by automorphisms on a quasi-split del Pezzo surface over any $k$ by Theorem \ref{Theorem:Quasi-split}.

    \item[Case 3:] Suppose $G'' \cong C_4$. Then $i \in k$ by Lemma \ref{Lemma:Pre-quasi}, and $G$ is a subgroup of $C_2^4 \rtimes \langle (1425) \rangle$ up to conjugacy. If $i \in k$, then $C_2^4 \rtimes \langle (1425) \rangle$ acts on a quasi-split quartic del Pezzo surface over $k$ by Theorem \ref{Theorem:Quasi-split}.

    \item[Case 4:] Suppose $G'' \cong C_5$ or $G'' \cong D_5$. Then $\sqrt{5} \in k$ by Lemma \ref{Lemma:Pre-quasi}, and $G$ is a subgroup of $C_2^4 \rtimes \langle (12)(45), (15342) \rangle$. If $\sqrt{5} \in k$, then $C_2^4 \rtimes \langle (12)(45), (15342) \rangle$ acts on a quasi-split quartic del Pezzo surface over $k$ by Theorem \ref{Theorem:Quasi-split}.

    \item[Case 5:] Suppose $G'' \cong C_3$. Then the exact sequence
    \begin{equation*}
        1 \to G' \to G \to G'' \to 1
    \end{equation*}
    splits by the Schur-Zassenhaus theorem. Up to conjugacy in $W(\mathsf{D}_5)$, we can assume $G'' = \langle (123) \rangle$, and $G$ is a subgroup of $C_2^4 \rtimes \langle (123) \rangle$. If $\epsilon_3 \in k$, then $C_2^4 \rtimes \langle (123) \rangle$ acts by automorphisms on a quasi-split quartic del Pezzo surface over $k$ by Theorem \ref{Theorem:Quasi-split}. Now suppose $\epsilon_3 \not\in k$. By Lemma \ref{Lemma:Pre-quasi}, $G'$ is a subgroup of $\langle c_1,c_2,c_3 \rangle$. The only non-trivial subgroups of $\langle c_1,c_2,c_3 \rangle$ invariant under the action of $\langle (123) \rangle$ are $\langle c_4c_5 \rangle$, $\langle c_1c_2,c_2c_3 \rangle$, and $\langle c_1,c_2,c_3 \rangle$. An element in $W(\mathsf{D}_5)$ of the form $a(123)$ with $a \in C_2^4$ has order 3 if and only if $a \in \langle c_1c_2,c_2c_3 \rangle$. Since the sequence splits, $G$ contains an element $a(123)$ of order 3. So if $G' = \langle c_1c_2,c_2c_3 \rangle$ or $G' = \langle c_1,c_2,c_3 \rangle$, then $G$ contains $(123)$. Therefore, $G$ is either $\langle c_1c_2,c_2c_3 \rangle \langle (123) \rangle$ or $\langle c_1,c_2,c_3 \rangle \langle (123) \rangle$. In either case, $G$ is a subgroup of $\langle c_1,c_2,c_3 \rangle \rtimes \langle (12)(45), (123) \rangle$ which acts on a quasi-split quartic del Pezzo surface over $k$ by Theorem \ref{Theorem:Quasi-split}. If $G' = \langle c_4c_5 \rangle$, then $G$ is one of the following groups:
    \begin{equation*}
        \langle c_4c_5(123) \rangle, \quad \langle c_1(123) \rangle, \quad \langle c_2(123) \rangle, \quad \langle c_3(123) \rangle,
    \end{equation*}
    but these groups are all conjugate in $W(\mathsf{D}_5)$. Consequently, we can assume that $G = \langle c_4c_5(123) \rangle$, which is contained in $\langle c_1,c_2,c_3 \rangle \rtimes \langle (12)(45), (123) \rangle$.
\end{itemize}

\noindent We conclude that any subgroup $G$ of $W(\mathsf{D}_5)$ that only acts on non-quasi-split surfaces must have $G'' \cong S_3$.
\end{proof}

\begin{thm}
    \label{Theorem:Non-quasi}
    Let $k$ be a field of characteristic zero. Suppose $G$ is a subgroup of $W(\mathsf{D}_5)$ that acts by automorphisms on a quartic del Pezzo surface over $k$ but does not act by automorphisms on any quasi-split quartic del Pezzo surface over $k$. Then, up to conjugacy, $G$ is one of the following groups:
    \begin{equation*}
        \langle c_1,c_2,c_3, (123), c_4(12)(45) \rangle = C_2^3 \cdot S_3 \quad \text{or} \quad \langle c_4c_5, (123), c_4(12)(45) \rangle = C_2 \cdot S_3
    \end{equation*}
\end{thm}

\begin{proof}
We have an exact sequence
\begin{equation}
    1 \to G' \to G \to G'' \to 1
\end{equation}
where $G''$ is the image of $G$ in $S_5$. By Lemma \ref{Lemma:ReducetoS3}, $G'' \cong S_3$. Up to conjugacy in $W(\mathsf{D}_5)$ we can assume that $G'' = \langle (12)(45), (123) \rangle$ and $G$ is a subgroup of $C_2^4 \rtimes \langle (12)(45), (123) \rangle$. If $\epsilon_3 \in k$, then $C_2^4 \rtimes \langle (12)(45), (123) \rangle$ acts on a quasi-split quartic del Pezzo surface over $k$ by Theorem \ref{Theorem:Quasi-split}, so we assume that $\epsilon_3 \not\in k$. Then by Lemma \ref{Lemma:Pre-quasi}, $G'$ is a subgroup of $\langle c_1,c_2,c_3 \rangle$. We have an exact sequence
\begin{equation}
\label{Split}
    1 \to G' \to H \to \langle (123) \rangle \to 1
\end{equation}
where $H$ is the subgroup of $G$ with image in $\langle (123) \rangle$. This sequence splits by the Schur-Zassenhaus theorem. Since $G'$ is closed under the action of $\langle (123) \rangle$, $G'$ must be $\langle c_4c_5 \rangle$, $\langle c_1c_2,c_2c_3 \rangle$, or $\langle c_1,c_2,c_3 \rangle$. However, if $G'$ contains $\langle c_1c_2,c_2c_3 \rangle$, then $\{P_1,P_2\}$, $\{P_2,P_3\}$, and $\{P_1,P_3\}$ form $\Gamma$-stable subsets of the five $\bar{k}$-points of $S_X$, where $X$ is the quartic del Pezzo surface over $k$ on which $G$ acts. It follows that $P_1$, $P_2$, and $P_3$ are $\Gamma$-fixed, so $G' = \langle c_1,c_2,c_3 \rangle$. We conclude that either $G' = \langle c_4c_5 \rangle$ or $G' = \langle c_1,c_2,c_3 \rangle$. 
\begin{itemize}
    \item[Case 1:] Suppose $G' = \langle c_1,c_2,c_3 \rangle$. Since sequence \ref{Split} splits, $H$ contains an element $a(123)$ with $a \in C_2^4$ of order 3. Now $a(123)$ has order 3 if and only if $a \in \langle c_1c_2,c_2c_3 \rangle$. Since $\langle c_1c_2,c_2c_3 \rangle$ is a subgroup of $G'$, we conclude that $H$ contains $(123)$. Therefore, $H = \langle c_1,c_2,c_3, (123) \rangle$. Now $G$ also contains an element of the form $a(12)(45)$ for some $a \in C_2^4$. If $a \in \langle c_1,c_2,c_3 \rangle$, then $G = \langle c_1,c_2,c_3,(123),(12)(45) \rangle$. If $a \not\in \langle c_1,c_2,c_3 \rangle$, then $G = \langle c_1,c_2,c_3,(123),c_4(12)(45) \rangle$. If $G = \langle c_1,c_2,c_3,(123),(12)(45) \rangle$, then $G$ acts on a quasi-split surface by Theorem \ref{Theorem:Quasi-split}. Therefore, $G$ must be $\langle c_1,c_2,c_3, (123), c_4(12)(45) \rangle$ up to conjugacy.

    \item[Case 2:] Suppose $G' = \langle c_4c_5 \rangle$. We have seen in Lemma \ref{Lemma:ReducetoS3} that $H$ must be $\langle c_4c_5 (123) \rangle$ up to conjugacy in $W(\mathsf{D}_5)$. Now $G$ also contains an element of the form $a(12)(45)$ for some $a \in C_2^4$. If $a \in \langle c_4c_5 \rangle$, then $(12)(45) \in G$, and $G = \langle c_4c_5,(123),(12)(45) \rangle$. It is simple to check that the only other possibility for $G$ is $\langle c_4c_5, (123), c_4(12)(45) \rangle$. If $G = \langle c_4c_5,(123),(12)(45) \rangle$, then $G$ acts on a quasi-split quartic del Pezzo surface. Therefore, $G$ must be $\langle c_4c_5, (123), c_4(12)(45) \rangle$ up to conjugacy.
\end{itemize}
\end{proof}

\begin{cor}
\label{Corollary:Epsilon_3}
    Let $k$ be a field of characteristic zero with $\epsilon_3 \in k$. Then any subgroup of $W(\mathsf{D}_5)$ that acts by automorphisms on a quartic del Pezzo surface acts by automorphisms on a quasi-split del Pezzo surface.
\end{cor}

\begin{proof}
The groups in Theorem \ref{Theorem:Non-quasi} are contained in $C_2^4 \rtimes \langle (12)(45), (123) \rangle$, which acts by automorphisms on a quasi-split surface when $\epsilon_3 \in k$ by Theorem \ref{Theorem:Quasi-split}.
\end{proof}

\noindent Let $G$ be a subgroup of $W(\mathsf{D}_5)$ that only acts by automorphisms on non-quasi-split del Pezzo surfaces of degree 4 over $k$. By Proposition \ref{Prop:Skorobogatov}, $G$ must act on a twist of some quasi-split surface $X_0$ over $k$. Consider the exact sequence of $\Gamma$-groups
\begin{equation}
    1 \to A \to \Aut(\overline{X_0}) \to H \to 1
\end{equation}
where $H$ is identified with the subgroup of $\PGL_2(\bar{k})$ stabilizing the five points of $S_{X_0}$. By Lemma \ref{Lemma:Moving} this sequence splits, and $\Aut(X_0) \cong A^{\Gamma} \rtimes H^{\Gamma}$. If we twist $X_0$ be a cocycle $(a_\gamma) \in Z^1(\bar{k}/k,A)$, then the $\Gamma$-action on $A$ and $H$ is unchanged. We obtain an exact sequence of $\Gamma$-groups
\begin{equation}
    1 \to A \to \Aut(\overline{{}^a X_0}) \to H \to 1
\end{equation}
that does not necessarily split.

\begin{lem}
\label{Lemma:Kernel}
    Let $A$ be a $\Gamma$-group and suppose $a \in Z^1(\Gamma, A)$. Then $\ker(a) = a^{-1}(\id)$ is a subgroup of $\Gamma$. Moreover, if $\ker(a)$ acts trivially on $\text{im}(a) \subseteq A$, then $\ker(a)$ is a normal subgroup. If $A$ is finite, then $\ker(a)$ is open and closed.
\end{lem}

\begin{proof}
Recall that a cocycle $a \in Z^1(\Gamma,A)$ is a function $a: \Gamma \to A$ satisfying the condition $a(\gamma_1 \gamma_2) = a(\gamma_1){}^{\gamma_1} a(\gamma_2)$. It is simple to prove, using the cocycle condition, that $\ker(a)$ is a subgroup of $\Gamma$. Now suppose $\ker(a)$ acts trivially on $\text{im}(a)$. For $\gamma \in \Gamma$ and $\sigma \in \ker(a)$, 
\begin{equation*}
    a(\gamma \sigma \gamma^{-1}) = a(\gamma \sigma){}^{\gamma \sigma} a(\gamma^{-1}) = a(\gamma) {}^\gamma a(\sigma) {}^\gamma a(\gamma^{-1}) = a(\gamma) {}^\gamma a(\gamma^{-1}) = a(\gamma\gamma^{-1}) = \id
\end{equation*}
So $\gamma\sigma\gamma^{-1} \in \ker(a)$, and $\ker(a)$ is a normal subgroup. Finally, suppose $A$ is finite and endowed with the discrete topology. Then $a$ is continuous, so $\ker(a) = a^{-1}(\id)$ must be an open and closed set.
\end{proof}

\noindent The following proposition offers a precise condition on $k$ for the groups in Theorem \ref{Theorem:Non-quasi} to act on a quartic del Pezzo surface over $k$. We also describe the surfaces on which the groups can act.

\begin{prop}
    \label{Prop:TwistCondition}
    Let $G$ be $\langle c_1,c_2,c_3,(123),c_4(12)(45) \rangle$ or $\langle c_4c_5,(123),c_4(12)(45) \rangle$.
    \begin{itemize}
        \item[(i)] $G$ acts on a quartic del Pezzo surface over $k$ if and only if $x^2 + y^2 = -3$ has a solution over $k$.

        \item[(ii)] Suppose $\epsilon_3 \not\in k$. Let $X$ be a quartic del Pezzo surface with an action of $G$. Then $X$ is a twist of a quasi-split surface $X_0$ by a cocycle $(a_\gamma) \in Z^1(\bar{k}/k,A)$, and $\Aut(X_0) = \langle c_1,c_2,c_3 \rangle \rtimes \langle (12)(45), (123) \rangle$ or $\Aut(X_0) = \langle c_4c_5 \rangle \rtimes \langle (12)(45), (123) \rangle$.  In either case, $(a_\gamma)$ surjects onto $\langle c_4,c_5 \rangle$ and satisfies the property
        \begin{align*}
            a_\gamma \in \begin{cases}
            \{\id, c_4c_5 \}, &\text{iff } \gamma(\epsilon_3) = \epsilon_3 \\
            \{c_4, c_5\}, &\text{iff } \gamma(\epsilon_3) = \epsilon_3^2
            \end{cases}.
        \end{align*}
    \end{itemize}
\end{prop}

\begin{proof}
Let $G$ be either of the groups mentioned in the proposition, and suppose $G$ acts on a quartic del Pezzo surface over $k$. If $\epsilon_3 \in k$, then $G$ acts on a quasi-split surface $X_0$ with $\Aut(X_0) = C_2^4 \rtimes \langle (12)(45), (123) \rangle$ by Theorem \ref{Theorem:Quasi-split}. Notice that $(1+2\epsilon_3)^2 + 0^2 = -3$ so that $x^2 + y^2 = -3$ has a solution over $k$. Now suppose $\epsilon_3 \not\in k$. We know $G$ sits in an exact sequence
\begin{equation*}
    1 \to G' \to G \to G'' \to 1
\end{equation*}
with $G' = \langle c_1,c_2,c_3 \rangle$ or $G' = \langle c_4c_5 \rangle$ and $G'' = \langle (12)(45), (123) \rangle$. Suppose $G$ acts on a twist of a quasi-split surface $X_0$ by a cocycle $(a_\gamma) \in Z^1(\bar{k}/k,A)$. Then $\Aut(X_0)$ sits in an exact sequence
\begin{equation*}
    1 \to A^{\Gamma} \to \Aut(X_0) \to \langle (12)(45), (123) \rangle \to 1
\end{equation*}
since twisting will not change the $\Gamma$-action on $A$ and $H$. Since $\epsilon_3 \not\in k$, $A^{\Gamma}$ is a subgroup of $\langle c_1,c_2,c_3 \rangle$ by Lemma \ref{Lemma:Pre-quasi}. We deduce that $A^{\Gamma}$ is $\langle c_1,c_2,c_3 \rangle$ or $\langle c_4c_5 \rangle$. Notice that the $\Gamma$-action on $\Aut({}^a X_0)$ must fix $(123)$. Then $a_\gamma$ must be in the centralizer of $(123)$ in $A$ for all $\gamma$. This forces $\im(a) \subseteq \langle c_4,c_5 \rangle$. Now $\langle (12)(45), (123) \rangle$ corresponds to a subgroup $H^{\Gamma}$ of $\PGL_2(k)$ stabilizing the five points of $S_{X_0}$. By Proposition \ref{Prop:Beauville}, we can assume the fixed points of $H^{\Gamma}$ up to projective equivalence are $(-\epsilon_3:1)$ and $(1:-\epsilon_3)$. Then $c_4$ and $c_5$ in $A$ are swapped by $\gamma \in \Gamma$ if and only if $\gamma(\epsilon_3) = \epsilon_3^2$ and fixed by $\gamma$ if and only if $\gamma(\epsilon_3) = \epsilon_3$. Now $G$ also contains $c_4(12)(45)$, which forces
\begin{align*}
    a_\gamma \in \begin{cases}
    \{\id, c_4c_5 \}, &\text{iff } \gamma(\epsilon_3) = \epsilon_3 \\
    \{c_4, c_5\}, &\text{iff } \gamma(\epsilon_3) = \epsilon_3^2
    \end{cases}.
\end{align*}
Suppose $\gamma(\epsilon_3) = \epsilon_3^2$ and $a_\gamma = c_4$, then $a_{\gamma^2} = c_4c_5$ and $a_{\gamma^3} = c_5$. Appealing to similar calculations if $a_\gamma = c_5$, we conclude that $a$ surjects onto $\langle c_4,c_5 \rangle$. This proves part (ii) of the proposition. Now notice that $\ker(a)$ acts trivially on $A$, so $\ker(a)$ is an open and closed normal subgroup of $\Gamma$ by Lemma \ref{Lemma:Kernel}. So $\ker(a) = \Gal(\bar{k} / F)$ for some finite field extension $F / k$. Since $\epsilon_3 \not\in k$, let $\psi: \Gamma \to \Gal(k(\epsilon_3) / k)$ be the restriction map. Then $\ker(\psi) = \Gal(\bar{k} / k(\epsilon_3))$. Notice that $\ker(a) \subset \ker(\psi)$, so $\bar{k} \supset F \supset k(\epsilon_3) \supset k$ is a chain of field extensions. We define a homomorphism $\Phi: \Gal(\bar{k}/k) \to \mathbb{Z}/4\mathbb{Z}$ by
\begin{equation*}
    \Phi(\gamma) = \begin{cases}
    0, \quad &\text{if } a_{\gamma} = \id \\
    1, \quad &\text{if } a_{\gamma} = c_4 \\
    2, \quad &\text{if } a_{\gamma} = c_4c_5 \\
    3, \quad &\text{if } a_{\gamma} = c_5
    \end{cases}
\end{equation*}
It is clear that $\ker(\Phi) = \ker(a)$, so 
\begin{equation*}
    \Gal(F/k) \cong \frac{\Gal(\bar{k}/k)}{\ker(a)} \cong \mathbb{Z}/4\mathbb{Z}.
\end{equation*}
Then $F/k$ is a $C_4$-extension containing $k(\sqrt{-3})$. By Theorem 2.2.5 of \cite{Jensen}, the field $k(\sqrt{-3})/k$ can be embedded in a $C_4$-extension of $k$ if and only if $x^2 + y^2 = -3$ has a solution over $k$.

\medskip

For the reverse direction, suppose $x^2+y^2=-3$ has a solution over $k$. If $\epsilon_3 \in k$, then $C_2^4 \rtimes \langle (12)(45), (123) \rangle$ acts on a quasi-split quartic del Pezzo surface over $k$ by Theorem \ref{Theorem:Quasi-split}, and we are done. Now suppose $\epsilon_3 \not\in k$. Let $F \supset k(\epsilon_3) \supset k$ be field extensions with $\Gal(F/k) \cong C_4$, which exist by Theorem 2.2.5 of \cite{Jensen}. Let $\Gal(F/k) = \langle \sigma \rangle$. We define a function $\pi: \Gal(F/k) \to \langle c_4, c_5 \rangle$ to be
\begin{equation*}
    \pi(\sigma^i) = \begin{cases}
    \id, \quad &\text{if } i = 0 \\
    c_4, \quad &\text{if } i = 1 \\
    c_4c_5, \quad &\text{if } i = 2 \\
    c_5, \quad &\text{if } i=3
    \end{cases}
\end{equation*}
Then we define the cocycle $(a_\gamma) \in Z^1(\bar{k}/k, A)$ to be the composition of the maps
\begin{equation*}
    \Gamma \to \Gal(F/k) \xrightarrow{\pi} \langle c_4, c_5 \rangle
\end{equation*}
where the first map is restriction. Notice that $\sigma(\epsilon_3) = \epsilon_3^2$ and $\sigma^2(\epsilon_3) = \epsilon_3$. Let $X_0$ be a quasi-split surface with $\Aut(X_0) = \langle c_1,c_2,c_3 \rangle \rtimes \langle (12)(45), (123) \rangle$. Then the twisted surface ${}^a X_0$ has $\Aut({}^a X_0) = \langle c_1,c_2,c_3,(123),c_4(12)(45) \rangle$.
\end{proof}

\begin{cor}
\label{Corollary:Quasi}
Let $k$ be a field of characteristic zero. If $x^2+y^2 = -3$ does not have a solution over $k$, then every subgroup $G$ of $W(\mathsf{D}_5)$ acting by automorphisms on a quartic del Pezzo surface over $k$ acts by automorphisms on a quasi-split quartic del Pezzo surface over $k$.
\end{cor}

\begin{proof}
By Theorem \ref{Theorem:Non-quasi}, we know which groups in $W(\mathsf{D}_5)$ can potentially act on only non-quasi-split surfaces. By Proposition \ref{Prop:TwistCondition}, these groups act on a quartic del Pezzo surface if and only if $x^2+y^2=-3$ has a solution over $k$.
\end{proof}

\noindent We are prepared to prove the main theorem. For any field $k$ of characteristic zero, recall that $\mathcal{M}_k$ denotes the collection of subgroups, up to conjugacy, in $W(\mathsf{D}_5)$ that act by automorphisms on a quartic del Pezzo surface over $k$. The following theorem determines the groups in $\mathcal{M}_k$ for any $k$.

\begin{thm}
    \label{Theorem:Main}
    Let $k$ be a field of characteristic zero. If $G$ is a maximal group in $\mathcal{M}_k$, then $G$ is one of the groups in the table. Each group appears in $\mathcal{M}_k$ if and only if $k$ satisfies the condition in the second column.
    \begin{center}
    \begin{tabular}{|c|c|}
        \hline
        \textbf{Group} & \textbf{Condition on } $k$ \\
        \hline
        $C_2^4 \rtimes C_4$ & $i \in k$ \\
        \hline
        $C_2^4 \rtimes S_3$ & $\epsilon_3 \in k$ \\
        \hline
        $C_2^4 \rtimes D_5$ & $\sqrt{5} \in k$ \\
        \hline
        $C_2^4 \rtimes C_2$ & none \\
        \hline
        $C_2^3 \rtimes S_3$ & none \\
        \hline
        $C_2^3 \cdot S_3$ & $x^2 + y^2 = -3$ has a $k$-point \\
        \hline
    \end{tabular}
    \end{center}
\end{thm}

\begin{proof}
The fact that $C_2^3 \cdot S_3$ is in $\mathcal{M}_k$ if and only if $x^2 + y^2 = -3$ has a $k$-point is the content of Proposition \ref{Prop:TwistCondition}. The conditions for the first five groups are verified in Theorem \ref{Theorem:Quasi-split}. By Corollary \ref{Corollary:Epsilon_3}, if $\epsilon_3 \in k$, then $\mathcal{M}_k = \mathcal{M}_k^{qs}$. By Corollary \ref{Corollary:Quasi}, if $x^2 + y^2 = -3$ does not have a solution over $k$, then $\mathcal{M}_k = \mathcal{M}_k^{qs}$. Finally, suppose $x^2 + y^2 = -3$ has a solution over $k$ and $\epsilon_3 \not\in k$. Every group in $\mathcal{M}_k$ is also in $\mathcal{M}_k^{qs}$ except for $\langle c_1,c_2,c_3,(123),c_4(12)(45) \rangle$ and $\langle c_4c_5,(123),c_4(12)(45) \rangle$ by Theorem \ref{Theorem:Non-quasi}. Both groups are contained in $\mathcal{M}_k$ by Proposition \ref{Prop:TwistCondition}, so $\langle c_1,c_2,c_3,(123),c_4(12)(45) \rangle = C_2^3 \cdot S_3$ is maximal in $\mathcal{M}_k$.
\end{proof}

\begin{cor}
    \label{cor:quadratic}
    Let $k$ be a field of characteristic zero. Let $G \subseteq W(\mathsf{D}_5)$ act by automorphisms on a quartic del Pezzo surface over $\bar{k}$. There exists a quadratic extension $F/k$ and a quasi-split quartic del Pezzo surface $X_0$ over $F$ such that $G$ acts by automorphisms on $X_0$.
\end{cor}  

\section{Equations} 

There are six maximal conjugacy classes of $W(\mathsf{D}_5)$ that appear in Theorem \ref{Theorem:Main}. For each maximal class $G$, we now exhibit explicit equations for a del Pezzo surface of degree 4 with an action of $G$ over the appropriate $k$. The machinery to find equations for these surfaces is contained in \cite{Flynn} and \cite{Skorobogatov}. Let $L$ be an \a'etale $k$-algebra so that $L = \bigoplus_{i=1}^m k_i$ for some finite separable field extensions $k_i/k$. The trace map $\Tr_{L/k}: L \to k$ is defined to be the sum of the trace maps $\Tr_{k_i/k} : k_i \to k$. Let $n = \dim_k L$. Notice that $k[x] / (P(x))$ is an $n$-dimensional \a'etale $k$-algebra when $P(x)$ is a separable polynomial of degree $n$. We define $\mu_n: \mathbf{alg_k} \to \mathbf{Group}$ to be the group $k$-scheme given by $R \mapsto \{r \in R \mid r^n = 1\}$. Assume that $n$ is odd. Consider the finite \a'etale abelian group $k$-scheme $A = R_{L/k}(\mu_2) / \mu_2$ where $R_{L/k}$ is the Weil restriction of scalars. Then $A(\bar{k}) \cong C_2^{n-1}$ is generated by $c_i$ for $1 \leq i \leq n$ with $\prod_{i=1}^n c_i = \id$. The $\Gamma$-action on $A(\bar{k})$ permutes the $c_i$ in the same way as the components of $L \otimes_k \bar{k}$. There is an exact sequence of $k$-groups
\begin{equation*}
    1 \to \mu_2 \to R_{L/k}(\mu_2) \to A \to 1.
\end{equation*}
As stated in \cite{Skorobogatov}, the induced map
\begin{equation*}
    H^2(k,\mu_2) \to H^2(k,R_{L/k}(\mu_2)) = H^2(L,\mu_2)
\end{equation*}
is injective. As a result
\begin{equation*}
    H^1(k,A) = L^*/k^*{L^*}^2 = \text{Coker}\left[\triangle : k^*/{k^*}^2 \to \prod_{i} k_i^* / {k_i^*}^2 \right].
\end{equation*}
In particular, every $\lambda \in L^*$ corresponds to a class $[\lambda] \in H^1(k,A)$.

\medskip

Recall that $S_X$ is the reduced degree five subscheme of $\mathbb{P}^1_k$ corresponding to any del Pezzo surface $X$ of degree 4. By choosing an appropriate affine subspace of $\mathbb{P}^1_k$, we can ensure that $S_X$ induces a separable monic polynomial $P(x) \in k[x]$. Then $L= k[x]/(P(x))$ is an \'etale algebra of dimension five over $k$. We let $\theta$ be the image of $x$ in $L$. Then $\mathbb{P}^4_k = \mathbb{P}(R_{L/k}\mathbb{A}^1_L)$, and we let $u = \sum_{i=0}^4 u_i\theta^i$ be a variable in $\mathbb{A}_L^1$. By Section 2 of \cite{Skorobogatov}, if we start with a reduced degree five subscheme of $\mathbb{P}^1_k$, the equations for the corresponding quasi-split surface are given by
\begin{equation}
\label{Eq:Quasi-split}
    \Tr_{L/k}(P'(\theta)^{-1}u^2) = \Tr_{L/k}(P'(\theta)^{-1}\theta u^2) = 0.
\end{equation}
Given an element $\lambda \in L^*$, the twist of the quasi-split surface by the class $[\lambda] \in H^1(k,A)$ is given by the equations
\begin{equation}
\label{Eq:DP4}
    \Tr_{L/k}(\lambda P'(\theta)^{-1} u^2) = \Tr_{L/k}(\lambda P'(\theta)^{-1} \theta u^2) = 0.
\end{equation}
If $G$ is $C_2^4 \rtimes C_2$, $C_2^4 \rtimes C_4$, or $C_2^4 \rtimes S_3$, equations for a quartic del Pezzo surface over $\mathbb{C}$ with an action of $G$ can be found in Section 8.6.4 of \cite{D12}. We can simply adopt these equations for the cases contained in the following example.

\begin{example}[$C_2^4 \rtimes C_2$, $C_2^4 \rtimes C_4$, $C_2^4 \rtimes S_3$]
    For arbitrary $k$, let $X$ be a surface given by
    \begin{equation*}
        \sum_{i=0}^4 u_i^2 = u_0^2 + au_1^2 - u_2^2 - au_3^2 = 0
    \end{equation*}
    with $a \neq 0, \pm 1$. For general $a \in k$, $\Aut(X) = C_2^4 \rtimes C_2$. If $i \in k$, let $X$ be the surface given by
    \begin{equation*}
        \sum_{i=0}^4 u_i^2 = u_0^2 + iu_1^2 - u_2^2 - iu_3^2 = 0.
    \end{equation*}
    Then $\Aut(X) = C_2^4 \rtimes C_4$ where $C_4$ is generated by the automorphism
    \begin{equation*}
        g_1: (u_0:...:u_4) \mapsto (u_1:u_2:u_3:u_0:u_4).
    \end{equation*}
    If $\epsilon_3 \in k$, let $X$ be the surface given by
    \begin{equation*}
        u_0^2 + \epsilon_3 u_1^2 + \epsilon_3^2 u_2^2 + u_3^2 = u_0^2 + \epsilon_3^2 u_1^2 + \epsilon_3 u_2^2 + u_4^2 = 0.
    \end{equation*}
    Then $\Aut(X) = C_2^4 \rtimes S_3$ where $S_3$ is generated by the automorphisms
    \begin{align*}
        &g_1: (u_0:...:u_4) \mapsto (u_0:u_2:u_1:u_4:u_3) \\
        &g_2: (u_0:...:u_4) \mapsto (u_1:u_2:u_0:\epsilon_3 u_3: \epsilon_3^2 u_4).
    \end{align*}
\end{example}
\begin{flushright}
    $\square$
\end{flushright}

\noindent To find equations for surfaces with an action of $C_2^4 \rtimes D_5$, $C_2^3 \rtimes S_3$, and $C_2^3 \cdot S_3$ over the appropriate field, we use equations \ref{Eq:Quasi-split} and \ref{Eq:DP4}.

\begin{example}[$C_2^4 \rtimes D_5$]
    Suppose $\sqrt{5} \in k$. Let $\phi = (1+\sqrt{5})/2$. We start with the five points
    \begin{equation*}
        \{(-1:\phi+1), (\phi+1:-1), (1:1), (2:\phi-1), (\phi-1:2) \}
    \end{equation*}
    in $\mathbb{P}^1_k$. Using equation \ref{Eq:Quasi-split} and rescaling, the corresponding quasi-split surface is given by the quadrics
    \begin{align*}
        q_1 &= -(\sqrt{5} + 1)u_0^2 + (5\sqrt{5} - 11)u_1^2 - (\sqrt{5}+1)u_2^2 + (5\sqrt{5} - 11)u_3^2 - (8\sqrt{5}-24)u_4^2 \\
        q_2 &= (\sqrt{5}-1)u_0^2 - (2\sqrt{5}-4)u_1^2 - (\sqrt{5}+1)u_2^2 - (6\sqrt{5}-14)u_3^2 + (8a-16)u_4^2.
    \end{align*}
    Then $\Aut(X) = C_2^4 \rtimes D_5$ where $D_5$ is generated by the automorphisms
    \begin{align*}
        g_1: (u_0:...:u_4) &\mapsto \left(2u_4:(2\phi + 1)u_0: (\phi-1)u_3: \phi u_1: \frac{1}{2}(\phi+1)u_2 \right) \\
        g_2 : (u_0:...:u_4) &\mapsto \left((2-\phi)u_1:(\phi+1)u_0:u_2:2\phi u_4: \frac{1}{2}(\phi-1)u_3 \right).
    \end{align*}
    One can check that
    \begin{equation*}
        h_5(q_1) = -\phi(q_1+q_2); \quad 
        h_5(q_2) = \phi q_1 - (\phi+1) q_2; \quad
        h_2(q_1) = q_2; \quad 
        h_2(q_2) = q_1.
    \end{equation*}
\end{example}
\begin{flushright}
    $\square$
\end{flushright}

\begin{example}[$C_2^3 \rtimes S_3$]
    Suppose $\epsilon_3 \not\in k$. We start with the five points
    \begin{equation*}
        \{(2:1),(1:2),(1:-1),(-\epsilon_3:1),(1:-\epsilon_3) \}
    \end{equation*}
    in $\mathbb{P}^1_{\bar{k}}$. Using equation \ref{Eq:Quasi-split} and rescaling, the corresponding quasi-split surface is given by the quadrics
    \begin{align*}
        q_1 &= u_0^2 - 8u_1^2 + u_2^2 - 3u_3^2 + 6u_3u_4 + 6u_4^2 \\
        q_2 &= 2u_0^2 - 4u_1^2 - u_2^2 - 6u_3^2 - 6u_3u_4 + 3u_4^2.
    \end{align*}
    Since $\epsilon_3 \not\in k$, $\Aut(X) = C_2^3 \rtimes S_3$ where $S_3$ is generated by
    \begin{align*}
        g_1 : (u_0:...:u_4) &\mapsto \left( u_2: \frac{1}{2}u_0: 2u_1: u_4: -u_3-u_4 \right) \\
        g_2: (u_0:...:u_4) &\mapsto \left( 2u_1:\frac{1}{2}u_0:u_2:u_4:u_3 \right).
    \end{align*}
    and $C_2^3$ consists of the diagonal automorphisms sending $u_i \mapsto \pm u_i$ for $i \in \{0,1,2\}$. One can check that
    \begin{equation*}
        h_3(q_1) = -q_2; \quad
        h_3(q_2) = q_1 - q_2; \quad 
        h_2(q_1) = -q_2; \quad
        h_2(q_2) = -q_1.
    \end{equation*}
\end{example}
\begin{flushright}
    $\square$
\end{flushright}

\begin{example}[$C_2^3 \cdot S_3$]
\label{Ex:C23S3}
    Assume that $\epsilon_3 \not\in k$ and suppose $\alpha^2 + \beta^2 = -3$ for some $\alpha,\beta \in k$. We adopt the five $\bar{k}$-points and quasi-split surface $X$ from the previous example. Then
    \begin{equation*}
        L = \dfrac{k[x]}{(P(x))} \cong \frac{k[x]}{(x-2)} \oplus \frac{k[x]}{(x- 1/2)} \oplus \frac{k[x]}{(x+1)} \oplus \frac{k[x]}{(x^2-x+1)}.
    \end{equation*}
    Twisting $X$ by $\lambda = (1,1,1,bx + c)$ for $b,c \in k$ produces the surface with equations
    \begin{align*}
        q_1 &= u_0^2 - 8u_1^2 + u_2^2 - 3(c+2b)u_3^2 + 6(c-b)u_3u_4 + 3(b+2c)u_4^2 \\
        q_2 &= 2u_0^2 - 4u_1^2 - u_2^2 - 3(b+2c)u_3^2 - 6(2b+c)u_3u_4 + 3(c-b)u_4^2.
    \end{align*}
    If we let
    \begin{equation*}
        \lambda = \left(1,1,1, \left( -\frac{1}{3} + \frac{1}{9}\alpha \right)x - \frac{2}{9}\alpha \right)
    \end{equation*}
    then ${}^\lambda X$ is given by
    \begin{align*}
        q_1 &= u_0^2 - 8u_1^2 + u_2^2 + 2 u_3^2 + 2(1-\alpha)u_3u_4 - (\alpha+1)u_4^2 \\
        q_2 &= 2u_0^2 - 4u_1^2 - u_2^2 + (\alpha+1)u_3^2 + 4 u_3u_4 + (1-\alpha)u_4^2.
    \end{align*}
    Then $\Aut({}^\lambda X) = C_2^3 \cdot S_3$ is generated by the diagonal automorphisms sending $u_i \mapsto \pm u_i$ for $i \in \{0,1,2\}$ as well as
    \begin{align*}
        &g_1: (u_0:...:u_4) \mapsto \left(u_2: \frac{1}{2}u_0: 2u_1: u_4: -u_3 - u_4 \right) \\
        &g_2: (u_0:...:u_4) \mapsto \left( 2u_1: \frac{1}{2}u_0 : u_2: \frac{-\alpha-1}{\beta}u_3 - \frac{2}{\beta}u_4: \frac{\alpha -1}{\beta}u_3 + \frac{\alpha +1}{\beta}u_4 \right).
    \end{align*}
\end{example}
\begin{flushright}
    $\square$
\end{flushright}

\section{Rationality}

We can also determine, for all $k$ of characteristic zero, which subgroups of $W(\mathsf{D}_5)$ act by automorphisms on a $k$-rational quartic del Pezzo surface. We say that a surface $X$ is \textit{$k$-minimal} if every birational $k$-morphism $f: X \to X'$, where $X'$ is a regular surface, is an isomorphism. By Lemma 0.4 of \cite{M66}, $X$ is $k$-minimal if and only if every $\Gamma$-orbit of exceptional curves on $X$ contains at least one couple of intersecting lines. By the main results of \cite{I72}, a $k$-minimal quartic del Pezzo surface is not $k$-rational.

\begin{prop}
    \label{prop:rationalityqs}
    Let $X_0$ be a quasi-split del Pezzo surface of degree 4 over $k$. Then $X_0$ is $k$-rational.
\end{prop}

\begin{proof}
By a result of Swinnerton-Dyer \cite{SD70}, every del Pezzo surface of degree five over $k$ has a $k$-rational point. Combined with Theorem 29.4 of \cite{M86}, it follows that every del Pezzo surface of degree five is $k$-rational. $X_0$ is isomorphic to the blowup of $\mathbb{P}^2_k$ in a $\Gamma$-stable subset of five points in general position. The strict transform $C$ of the conic through the five points is defined over $k$. Blowing down along $C$, we obtain a del Pezzo surface of degree five, which must be $k$-rational.
\end{proof}

\begin{prop}
    \label{Prop:RationalityTwist}
    Suppose $\epsilon_3 \not\in k$, and suppose $X$ is a quartic del Pezzo surface over $k$ with an action of $C_2^3 \cdot S_3$ or $C_2 \cdot S_3$. Then $X$ is $k$-minimal, and therefore not $k$-rational.
\end{prop}

\begin{proof}
By Proposition \ref{Prop:TwistCondition}, since $\epsilon_3 \not\in k$, the surface $X$ is a twist of a quasi-split surface $X_0$ with $\Aut(X_0) = C_2^3 \rtimes S_3$ or $\Aut(X_0) = C_2 \rtimes S_3$. Let $X$ be isomorphic to ${}^a X_0$ for a cocycle $(a_\gamma) \in Z^1(\bar{k}/k,A)$. Then $(a_\gamma)$ surjects onto $\langle c_4,c_5 \rangle$ and satisfies the property
\begin{align*}
    a_\gamma \in \begin{cases}
    \{\id, c_4c_5 \}, &\text{iff } \gamma(\epsilon_3) = \epsilon_3 \\
    \{c_4, c_5\}, &\text{iff } \gamma(\epsilon_3) = \epsilon_3^2
    \end{cases}.
\end{align*}
Using Lemma \ref{lemma:Aaction}, we list the permutations of the exceptional divisors of $X_0$ induced by $c_4$, $c_5$, and $c_4c_5$:
\begin{align*}
    c_4 &: (E_1 \, L_{14}) (E_2 \, L_{24}) (E_3 \, L_{34}) (E_4 \, C) (E_5 \, L_{45}) (L_{12} \, L_{35}) (L_{13} \, L_{25}) (L_{23} \, L_{15}) \\
    c_5 &: (E_1 \, L_{15}) (E_2 \, L_{25}) (E_3 \, L_{35}) (E_4 \, L_{45}) (E_5 \, C) (L_{12} \, L_{34}) (L_{13} \, L_{24}) (L_{14} \, L_{23}) \\
    c_4c_5 &: (E_1 \, L_{23}) (E_2 \, L_{13}) (E_3 \, L_{12}) (E_4 \, E_5) (L_{14} \, L_{15}) (L_{24} \, L_{25}) (L_{34} \, L_{35}) (L_{45} \, C)
\end{align*}
The twisted $\Gamma$-action on the exceptional divisors of $X$ is defined by the rule ${}^{\gamma'}D = a_{\gamma} \cdot {}^\gamma D$. Then each $\Gamma$-orbit of exceptional divisors on $X$ is a union of one or more of the following sets:
\begin{align*}
    \{E_1, L_{23}, L_{14}, L_{15} \}, \quad \{E_2, L_{13}, L_{24}, L_{25} \}, \quad \{E_3, L_{12}, L_{34}, L_{35} \}, \quad \{E_4, E_5, L_{45}, C\}
\end{align*}
Notice that $L_{ij} \cdot L_{i'j'} = 1$ if and only if $i,j,i',j'$ are distinct, $E_i \cdot L_{jk} = 1$ if and only if $i=j$ or $i=k$, and $C \cdot E_i = 1$ for all $i$. Each set contains a pair of lines that intersect. We conclude that $X$ is $k$-minimal and therefore not $k$-rational.
\end{proof}

\noindent For any field $k$ of characteristic zero, we let $\mathcal{M}_k^{rat}$ denote the collection of subgroups, up to conjugacy, in $W(\mathsf{D}_5)$ that act by automorphisms on a $k$-rational quartic del Pezzo surface over $k$. We are able to describe $\mathcal{M}_k^{rat}$ for every $k$.

\begin{thm}
    \label{Theorem:Rationality}
    Let $k$ be a field of characteristic zero.
    \begin{itemize}
        \item[(i)] If $\epsilon_3 \in k$, then $\mathcal{M}_k = \mathcal{M}_k^{rat}$.

        \item[(ii)] If $x^2 + y^2 = -3$ does not have a solution over $k$, then $\mathcal{M}_k = \mathcal{M}_k^{rat}$.

        \item[(iii)] If $x^2+y^2 = -3$ has a solution over $k$ and $\epsilon_3 \not\in k$, then every group in $\mathcal{M}_k$ is in $\mathcal{M}_k^{rat}$ except the groups $C_2^3 \cdot S_3$ and $C_2 \cdot S_3$.
    \end{itemize}
\end{thm}

\begin{proof}
Parts (i) and (ii) are direct consequences of Corollary \ref{Corollary:Epsilon_3}, Corollary \ref{Corollary:Quasi}, and Proposition \ref{prop:rationalityqs}. Part (iii) is an easy consequence of Theorem \ref{Theorem:Non-quasi} and Proposition \ref{Prop:RationalityTwist}.
\end{proof}

\noindent One might also wonder if either $C_2^3 \cdot S_3$ or $C_2 \cdot S_3$ act on a stably $k$-rational quartic del Pezzo surface if $k$ does not contain $\epsilon_3$. Recall that a variety $X$ is stably $k$-rational if $X \times \mathbb{P}^m_k$ is $k$-rational for some $m$. For any quartic del Pezzo surface $X$ over $k$, the $\Gamma$-action on $X$ induces an action on $\Pic \overline{X}$ that preserves the intersection pairing and the canonical divisor. There is a corresponding map $\rho: \Gamma \to W(\mathsf{D}_5)$. We define the \textit{splitting group} of $X$ to be the image of $\Gamma$ in $W(\mathsf{D}_5)$ (see Section $0$ of \cite{KST89}). The following is Theorem 5.20 of \cite{KST89}:

\begin{thm}
    \label{Theorem:KST89}
    If a del Pezzo surface $X$ of degree 4 with $(\Pic \overline{X})^\Gamma \cong \mathbb{Z}$ is stably $k$-rational but not $k$-rational, then its splitting group $G_X$ is conjugate to one of the following groups:
    \begin{equation*}
        I_1 = \langle c_1c_5, c_1(23), (234) \rangle, \quad I_2 = \langle c_1(15)(23), (234) \rangle, \quad I_3 = \langle c_1(23), (15), (234) \rangle.
    \end{equation*}
\end{thm}

\begin{prop}
    \label{Prop:StablyRational}
    Suppose $\epsilon_3 \not\in k$, and suppose $X$ is a quartic del Pezzo surface over $k$ with an action of $C_2^3 \cdot S_3$ or $C_2 \cdot S_3$. Then $X$ is not stably $k$-rational.
\end{prop}

\begin{proof}
$X$ is a twist of a quasi-split surface $X_0$ with $\Aut(X_0) = \langle c_1,c_2,c_3,(123),(12)(45) \rangle$ or $\Aut(X_0) = \langle c_4c_5,(123),(12)(45) \rangle$. In either case, let $\gamma$ be an element of $\Gamma$ that swaps $c_4$ and $c_5$ in $\Aut(\overline{X_0})$ and fixes $\langle c_1,c_2,c_3,(123),(12)(45) \rangle$. Then the splitting group $G_{X_0}$ contains $(45)$. Let $a \in Z^1(\bar{k}/k,A)$ be a cocycle with $X = {}^a X_0$. Without loss of generality, we can assume $a_\gamma = c_4$ by Proposition \ref{Prop:TwistCondition}. Then the splitting group $G_X$ contains $c_4(45)$. With respect to the basis $\{E_1,E_2,E_3,E_4,E_5,H\}$, $c_4(45)$ acts on $\Pic \overline{X}$ via the matrix
\begin{equation*}
    \begin{bmatrix} -1 & 0 & 0 & 0 & -1 & -1 \\ 0 & -1 & 0 & 0 & -1 & -1 \\ 0 & 0 & -1 & 0 & -1 & -1 \\ -1 & -1 & -1 & -1 & -1 & -2 \\ 0 & 0 & 0 & -1 & -1 & -1 \\ 1 & 1 & 1 & 1 & 2 & 3 \end{bmatrix}.
\end{equation*}
One checks that $(\Pic \overline{X})^{\langle c_4(45) \rangle} \cong \mathbb{Z}$. It follows that $(\Pic \overline{X})^\Gamma \cong \mathbb{Z}$. Now since $C_2 \cdot S_3$ is fixed by the $\Gamma$-action, $G_X$ must be in the centralizer of $C_2 \cdot S_3$ in $W(\mathsf{D}_5)$. We compute that the centralizer of $C_2 \cdot S_3$ is $\langle c_4(45) \rangle$. None of the $G_X$ in Theorem \ref{Theorem:KST89} are contained in $\langle c_4(45) \rangle$, so $X$ is not stably $k$-rational.
\end{proof}

\noindent The proof of Proposition \ref{Prop:StablyRational} points to an interesting observation.

\begin{prop}
    \label{Prop:AutsofStab}
    Let $X$ be a quartic del Pezzo surface over $k$ with $(\Pic \overline{X})^{\Gamma} \cong \mathbb{Z}$ that is stably $k$-rational but not $k$-rational. Then, up to conjugacy, $\Aut(X)$ is a subgroup of $\langle c_1,c_5 \rangle$ or $\langle c_5(15) \rangle$.
\end{prop}

\begin{proof}
If $X$ is stably $k$-rational but not $k$-rational, then $G_X$ is one of the groups in Theorem \ref{Theorem:KST89} up to conjugacy. The image of $\Aut(X)$ in $W(\mathsf{D}_5)$ must be in the centralizer of $G_X$. The centralizer of each of the possibilities for $G_X$ is contained in $\langle c_1,c_5 \rangle$ or $\langle c_5(15) \rangle$.
\end{proof}

\begin{remark}
    Using Proposition \ref{Prop:StablyRational}, we could rewrite Theorem \ref{Theorem:Rationality} with respect to stable $k$-rationality instead of $k$-rationality.
\end{remark} 

\noindent Lastly, we address the $k$-unirationality of surfaces with various geometric actions. Recall that a variety $X$ is $k$-unirational if there is a dominant rational map $f: \mathbb{P}_k^m \dashrightarrow X$ defined over $k$. It is clear that quasi-split surfaces are $k$-unirational. The following result is due to Manin (see Theorems 29.4 and 30.1 of \cite{M86}):

\begin{thm}
    Let $X$ be a del Pezzo surface of degree $d \geq 3$ over a field $k$ of characteristic zero. Then $X$ is $k$-unirational if and only if $X$ has a rational $k$-point.
\end{thm}

\begin{prop}
    \label{Prop:Unirationality}
    Let $k$ be a field of characteristic zero. Every group in $W(\mathsf{D}_5)$ that acts by automorphisms on a quartic del Pezzo surface over $k$ acts by automorphisms on a $k$-unirational quartic del Pezzo surface over $k$.
\end{prop}

\begin{proof}
Since every quasi-split surface is $k$-unirational, we only need to consider the groups $C_2^3 \cdot S_3$ and $C_2 \cdot S_3$. These two groups act on a quartic del Pezzo surface if and only if $x^2+y^2 = -3$ has a solution over $k$. Under this assumption, we constructed a surface with an action of $C_2^3 \cdot S_3$ in Example \ref{Ex:C23S3}. Notice that this surface contains the point $(2:1:2:0:0)$ and is consequently $k$-unirational.
\end{proof}

\bibliographystyle{alpha}
\bibliography{mybibliography}

\end{document}